\documentclass[3p]{elsarticle}

\usepackage{amsfonts}
\usepackage{amsmath}
\usepackage{amsthm}
\usepackage{amssymb}

\newcommand{\Z}{\mathbb{Z}}

\newcommand{\R}{\mathbb{R}}
\newcommand{\C}{\mathbb{C}}

\theoremstyle{plain} 
\newtheorem{theorem}{Theorem}[section]
\newtheorem{corollary}[theorem]{Corollary}
\newtheorem{lemma}[theorem]{Lemma}

\newtheorem{proposition}[theorem]{Proposition}

\newtheorem*{lemma*}{Lemma}
\newtheorem*{claim*}{Claim}
\newtheorem*{theorem*}{Theorem}
\newtheorem*{proposition*}{Proposition}
\newtheorem*{problemstatement*}{Problem}
\newtheorem*{corollary*}{Corollary}

\theoremstyle{definition} 
\newtheorem{definition}[theorem]{Definition}

\newtheorem{notation}[theorem]{Notation}
\newtheorem*{notation*}{Notation}

\theoremstyle{remark} 

\newcommand{\abs}[1]{\left| #1 \right|}
\newcommand{\conj}[1]{\overline{#1}}

\newcommand{\norm}[1]{\lVert #1 \rVert }

\newcommand{\grad}{\nabla}
\newcommand{\evalat}[2]{\bigl. #1 \bigr|_{#2}\rule{0pt}{3ex}}
\renewcommand{\Re}{\operatorname{Re}}
\renewcommand{\Im}{\operatorname{Im}}
\newcommand{\Res}{\operatorname{Res}}

\newcommand{\ad}{\operatorname{ad}}
\newcommand{\spanop}{\operatorname{span}}
\usepackage{cancel}
\usepackage{hyperref}
\newcommand{\inner}[2]{\langle #1 , #2 \rangle}
\newcommand{\innerp}[2]{\left( #1, #2 \right)}
\usepackage{url}
\numberwithin{equation}{section}
\usepackage{paralist}
\usepackage{enumitem}
\newcounter{continuehere}

\newcommand{\firstder}[1]{\frac{\partial}{\partial #1}}

\theoremstyle{remark}
\newtheorem*{remark*}{Remark}

\newcommand{\ignorethis}[1]{}

\newcommand{\polar}{\gamma}

\begin{document}

\begin{frontmatter}

\title{Precise Estimates for the Subelliptic Heat Kernel on H-type Groups}
\author{Nathaniel Eldredge}

\address{Department of Mathematics, University of California, San
  Diego, 
9500 Gilman Drive, Dept. 0112, La Jolla, CA 92093-0112 USA}
\ead{neldredge@math.ucsd.edu}
\ead[url]{http://www.math.ucsd.edu/~neldredg/}

 \begin{abstract}

   We establish precise upper and lower bounds for the 
   subelliptic heat kernel on nilpotent Lie groups $G$ of H-type.
   Specifically, we show that there exist positive constants $C_1,
   C_2$ and a polynomial correction function $Q_t$ on $G$ such that
   \begin{equation*}
     C_1 Q_t e^{-\frac{d^2}{4t}} \le p_t \le C_2 Q_t e^{-\frac{d^2}{4t}} 
   \end{equation*}
   where $p_t$ is the heat kernel, and $d$ the Carnot-Carath\'eodory
   distance on $G$.  We also obtain similar bounds on the norm of
   its subelliptic gradient $\abs{\grad p_t}$.  Along the way, we
   record explicit formulas for the distance function $d$ and the
   subriemannian geodesics of H-type groups.

    On donne des estimations pr\'ecises des bornes sup\'erieures et
    inf\'erieures du noyau de la chaleur sous-elliptique sur les
    groupes de Lie nilpotents $G$ de type H.  Plus pr\'ecis\'ement, on
    montre qu'il existe des constantes positives $C_1$ et $C_2$, et
    une fonction polynomiale corrective $Q_t$ sur $G$ telles que
   \begin{equation*}
     C_1 Q_t e^{-\frac{d^2}{4t}} \le p_t \le C_2 Q_t e^{-\frac{d^2}{4t}}, 
   \end{equation*}
   o\`u $p_t$ est le noyau de la chaleur, et $d$ est la distance de
   Carnot-Carath\'eodory sur $G$.  On obtient aussi des estimations
   similaires pour la norme du gradient $\abs{\nabla p_t}$.  En passant,
   on donne aussi des formules explicites pour la distance $d$ et les
   g\'eod\'esiques sous-riemannienes sur les groupes de type H.

\end{abstract}


\end{frontmatter}

\section{Introduction}

Nilpotent Lie groups have long been of interest as a natural setting
for the study of subelliptic operators; indeed, as shown in
\cite{rothschild-stein}, they model, at least locally, a general class
of hypoelliptic operators on manifolds.  Perhaps the simplest example
is the classical Heisenberg group of dimension $3$, followed by the
higher-dimensional Heisenberg or Heisenberg-Weyl groups of dimension
$2n+1$ having $1$-dimensional centers.  Beyond this, a natural
generalization of the Heisenberg groups is given by the H-type (or
Heisenberg-type) groups, which were introduced in \cite{kaplan80};
these have a greater variety of possible dimensions while retaining
some fairly strong algebraic structure.

The main result of this paper is found in Corollary
\ref{main-corollary}, in which we establish precise upper and lower
pointwise estimates on the subelliptic heat kernel $p_t$ for an H-type
group $G$, of the form
\begin{equation}\label{p-bound-intro}
  C_1 Q_t e^{-\frac{d^2}{4t}} \le p_t \le C_2 Q_t e^{-\frac{d^2}{4t}}
\end{equation}
for some positive constants $C_1$, $C_2$ and an explicit function
$Q_t$, where $d$ is the Carnot-Carath\'eodory distance on $G$.
Additionally, in Theorem \ref{main-gradient-theorem}, we obtain
similar bounds for the subriemannian gradient of $p_1$, namely that
\begin{equation}\label{grad-bound-intro}
  C_1 Q' e^{-\frac{d^2}{4}} \le \abs{\grad p_1} \le C_2 Q' e^{-\frac{d^2}{4}}
\end{equation}
for another explicit function $Q'$, where the inequality is valid at
points sufficiently far from the identity of $G$.  

Estimates of the form (\ref{p-bound-intro}) for the classical
Heisenberg group first appeared in \cite{li-jfa}, in the context of a
gradient estimate for the heat semigroup, as did an estimate
equivalent to the upper bound in (\ref{grad-bound-intro}).  A proof
for Heisenberg groups in all dimensions followed in
\cite{li-heatkernel}.  Our proof is similar in spirit to the latter,
in that it relies on the analysis of an explicit formula for $p_t$
using steepest descent methods and elementary complex analysis.

Less precise versions of the inequalities (\ref{p-bound-intro}) are
known to hold in more general settings.  Using Harnack
inequalities one can show that for general nilpotent Lie groups,
\begin{equation}
  C_1 R_1(t) e^{-\frac{d^2}{ct}} \le p_t \le C_2(\epsilon) R_2(t) e^{-\frac{d^2}{(4+\epsilon)t}}
\end{equation}
for some constants $c, C_1, C_2$ and functions $R_1, R_2$, where $C_2$
depends on $\epsilon > 0$; see chapter IV of \cite{purplebook}.
\cite{davies-pang}, among others, improves the upper bound to
\begin{equation}
   p_t(g) \le C R_3(g,t) e^{-\frac{d(g)^2}{4t}},
\end{equation}
with $R$ a polynomial correction, using logarithmic Sobolev
inequalities, whereas \cite{varopoulos-II} improves the lower bound to
\begin{equation}
  p_t \ge C(\epsilon) R_4(t) e^{-\frac{d^2}{(4-\epsilon)t}}. 
\end{equation}
Similar but slightly weaker estimates were shown for more general
sum-of-squares operators satisfying H\"ormander's condition in
\cite{kusuoka-stroock-III} by means of Malliavin calculus, and in
\cite{jerison-sanchez} by more elementary methods involving
homogeneity and the regular dependence of $p_t$ on $t$.

In the specific case of the classical Heisenberg group, asymptotic
results similar to (\ref{p-bound-intro}) had been previously obtained
in \cite{gaveau77} and \cite{hueber-muller}, but without the necessary
uniformity to translate them into pointwise estimates.  A precise
upper bound equivalent to that of (\ref{p-bound-intro}) was given in
\cite{bgg} for Heisenberg groups of all dimensions.  All three of
these works, like \cite{li-heatkernel} and the present article, were
based on an explicit formula for $p_t$ and involved steepest descent
type methods.  In \cite{garofalo-segala}, similar techniques were used
to obtain a Li-Yau-Harnack inequality for the heat equation on
Heisenberg groups.

The proof we shall give here is largely self-contained, except for the
formula (\ref{Rm-integral}) for $p_t$, which has been derived many
times in the literature by many different techniques.  We have also
tried to err to the side of including relevant details.

The author would like to extend his deep gratitude to his advisor,
Bruce Driver, for his constant feedback, encouragement and support
during the preparation of this paper.  The author would also like to
thank M. Salah Baouendi who provided some French linguistic advice,
and the anonymous referee who suggested some useful references,
especially \cite{garofalo-segala}.
This research was supported in part by NSF Grants DMS-0504608 and
DMS-0804472, as well as an NSF Graduate Research Fellowship.


\section{H-type groups}

H-type groups were first introduced in \cite{kaplan80}.  Chapter 18 of
 \cite{blu-book} contains an extended development of their fundamental
 properties; we follow its definitions here, and refer the reader
 there for further details.

\begin{definition}\label{H-type}
  Let $\mathfrak{g}$ be a finite dimensional real Lie algebra with
  center $\mathfrak{z} \ne 0$.  We say $\mathfrak{g}$ is of
  \emph{H-type} (or \emph{Heisenberg type}) if $\mathfrak{g}$ is
  equipped with an inner product $\inner{\cdot}{\cdot}$ such that:
  \begin{enumerate}
  \item $[\mathfrak{z}^\perp, \mathfrak{z}^\perp] = \mathfrak{z}$ ;
    and
  \item For each $z \in \mathfrak{z}$, define $J_z :
    \mathfrak{z}^\perp \to \mathfrak{z}^\perp$ by 
\begin{equation}\label{Jz-def}
\inner{J_z x}{y} =
    \inner{z}{[x,y]}
\end{equation}
 where $x, y \in \mathfrak{z}^\perp$.  Then
    $J_z$ is an orthogonal map whenever $\inner{z}{z}=1$.
  \end{enumerate}
 
  An \emph{H-type group} is a connected, simply connected Lie group
  whose Lie algebra is of H-type.
\end{definition}

Some authors use instead of item 2 the equivalent property that
for $x \in \mathfrak{z}^\perp$ with $\norm{x}=1$, the map $\ad_x : (\ker \ad_x)^\perp \to
\mathfrak{z}$ is an isometric isomorphism.

We record some algebraic properties of the maps $J_z$
which will be useful later.  We use $\abs{z} := \sqrt{\inner{z}{z}}$
to denote the norm associated to the inner product on $\mathfrak{g}$.
The proofs are elementary and are omitted.

\begin{proposition}\label{Jz-props}
  If $\mathfrak{g}$ is a H-type Lie algebra, then the maps $J_z :
  \mathfrak{z}^\perp \to \mathfrak{z}^\perp$ defined in Definition
  \ref{H-type} enjoy the following properties:
  \begin{enumerate}
    \item For each $z$, $J_z$ is a well-defined linear map, and $z
      \mapsto J_z$ is also linear. \label{Jz-linear}
    \item $J_z^* = -J_z$. \label{Jz-adjoint}
    \item $J_z^2 = -\norm{z}^2 I$.  Thus for $z \ne 0$, $J_z$ is
      invertible and $J_z^{-1} = -\norm{z}^{-2} J_z$. \label{Jz-square}
    \item $J_z J_w + J_w J_z = -2\inner{z}{w} I$. \label{Jz-clifford}
    \item $\inner{J_z x}{J_w x}  = \inner{z}{w} \norm{x}^2$. \label{Jz-inner}
    \item $[x, J_z x] = \norm{x}^2 z$. \label{Jz-bracket}
\end{enumerate}
\end{proposition}

Note that items \ref{Jz-adjoint} and \ref{Jz-square} say that for $z
\ne 0$, $J_z$ is an invertible skew-symmetric linear transformation of
$\mathfrak{z}^\perp$.  Thus $\dim \mathfrak{z}^\perp$ must be even.
We will write $\dim \mathfrak{z}^\perp = 2n$ and $\dim \mathfrak{z} =
m$.

Item \ref{Jz-clifford} says that the subalgebra of
$\operatorname{End}(\mathfrak{z}^\perp)$ generated by the maps $J_z$
is a Clifford algebra.  In fact, it is a $2n$-dimensional
representation of $C\ell_{0,m}(\R)$, the Clifford algebra  generated by
a real vector space of dimension $m$ with a negative definite
quadratic form (whose signature is $(0,m)$).  So in order for an
H-type algebra with $\dim \mathfrak{z} = m$, $\dim \mathfrak{z}^\perp
= 2n$ to exist, it is necessary that $C\ell_{0,m}(\R)$ have such a
representation.  This condition is also sufficient: given such a
representation, let $V$ be the $m$-dimensional generating subspace of
$C\ell_{0,m}(\R)$, and let $\mathfrak{g} = \R^{2n} \oplus V$, with the
maps $J_z$ defined by the representation.  Then the bracket on
$\mathfrak{g}$ can be recovered in terms of the $J_z$ from
\ref{Jz-def}, and $\mathfrak{g}$ is an H-type Lie algebra.

The Hurwitz-Radon-Eckmann theorem, as found in \cite{eckmann}, gives
necessary and sufficient conditions on $n$ and $m$ for such a
representation to exist.  The corresponding theorem for H-type
algebras appears as Corollary 1 of \cite{kaplan80}, which we quote
here.

\begin{theorem}\label{dimension-classification}
  For any nonnegative integer $k$, we can uniquely write $k=a
  2^{4p+q}$ where $a$ is odd and $0 \le q \le 3$; let $\rho(k) :=
  8p+2^q$.  ($\rho$ is sometimes called the Hurwitz-Radon function.)
  There exists an H-type Lie algebra 
  of dimension $2n+m$ with center of dimension $m$ if and only if $m <
  \rho(2n)$.  In particular, for every $m \in \mathbb{N}$ there exists an
  H-type Lie algebra with center of dimension $m$.
\end{theorem}

The special case $m=1$ gives the so-called isotropic Heisenberg groups
(also called the Heisenberg-Weyl groups) of real dimension $2n+1$; the
very special case $n=m=1$ is the classic Heisenberg group of dimension
$3$.

A Lie algebra $\mathfrak{g}$ is said to be \emph{nilpotent} of step
$k$ if $k$ is the smallest integer such that all $k$-fold brackets
of elements of $\mathfrak{g}$ vanish.  A nilpotent Lie algebra is
\emph{stratified} if we can write $\mathfrak{g} = \mathfrak{g}_1
\oplus \dots \oplus \mathfrak{g}_k$ where $[\mathfrak{g}_1,
  \mathfrak{g}_{i-1}] = \mathfrak{g}_i$ and $[\mathfrak{g}_1,
  \mathfrak{g}_k]=0$.  An H-type Lie algebra is
obviously stratified nilpotent of step $2$, with $\mathfrak{g}_1 =
\mathfrak{z}^\perp$, $\mathfrak{g}_2 = \mathfrak{z}$.

We recall that given a nilpotent Lie algebra $\mathfrak{g}$, there
exists a connected, simply connected Lie group $G$ whose Lie algebra
is $\mathfrak{g}$, and $G$ is unique up to isomorphism.  Indeed, we
can, and will, take $G$ to be $\mathfrak{g}$ equipped with the group
operation $\circ$ given by the Baker-Campbell-Hausdorff formula, which
for $\mathfrak{g}$ nilpotent of step $2$ reads
\begin{equation}\label{BCH-formula}
  x \circ y := x + y + \frac{1}{2}[x,y].
\end{equation}
In this case the exponential map $\mathfrak{g} \to G$ is just the
identity.  It is obvious, then, that if $\mathfrak{g}, \mathfrak{g}'$
are isomorphic as Lie algebras, then
$(\mathfrak{g},\circ),(\mathfrak{g}',\circ')$ as defined above are
isomorphic as Lie groups.

On the other hand, $\mathfrak{g}$ can be identified as an inner
product space with Euclidean space $\R^{2n+m}$, identifying
$\mathfrak{z}^\perp$ with the first $2n$ coordinates and
$\mathfrak{z}$ with the last $m$.  Therefore we can
handle H-type groups concretely as follows.

\begin{proposition}
  If $G$ is an H-type group, then there exist integers $n,m$ and a
  bracket operation $[\cdot,\cdot]$ on $\R^{2n+m}$ such that
  $(\R^{2n+m}, [\cdot,\cdot])$ is an H-type Lie algebra whose center
  is $\R^m$ and $G$ is isomorphic to $(\R^{2n+m}, \circ)$, where
  $\circ$ is defined by (\ref{BCH-formula}).
\end{proposition}

Henceforth we shall assume that any H-type group $G$ \emph{is} of this
form.  We shall use the notation \mbox{$g = (x,z) = (x^1, \dots, x^{2n},
z^1, \dots, z^m)$} to refer to points of $G$.  The identity of $G$ is
$(0,0)$, and the inverse operation is given by $(x,z)^{-1} = (-x,-z)$.
Because of the identification of $G$ with its Lie algebra, we will
view $[\cdot,\cdot]$ as a bracket on $G$.  By a slight abuse of
notation, we will also use $[\cdot,\cdot]$ to refer to the restriction
of $[\cdot,\cdot]$ to $\R^{2n} \oplus \R^{2n} \subset G \oplus G$,
which is a bilinear skew-symmetric mapping from $\R^{2n} \oplus
\R^{2n}$ to $\R^m$.  The maps $\{J_z : z \in \R^m\}$ are identified
with $2n\times 2n$ skew-symmetric matrices which are orthogonal when
$\abs{z}=1$.

We let $\{e_1, \dots, e_{2n}\}$ denote the standard basis for
$\R^{2n}$, and $\{u_1, \dots, u_m\}$ denote the standard basis for $\R^m$.

Note that the group operation on $G$ does not preserve the inner
product, and the vector space operations $g \mapsto g + h$, $g \mapsto
cg$ are not group homomorphisms of $G$.  However, the \emph{dilation}
\begin{equation}\label{dilation}
 \varphi_\alpha(x,z) := (\alpha x, \alpha^2 z) 
\end{equation}
is both a group and a Lie algebra automorphism for all $\alpha \ne 0$.

We can now identify $\mathfrak{g}$ with the set of left-invariant vector
fields on $G$, where $X_i(0) = \frac{\partial}{\partial x^i}$,
$Z_j(0)=\frac{\partial}{\partial z^j}$; then
$\spanop\{X_1, \dots, X_{2n}\} = \mathfrak{z}^\perp$, $\spanop\{Z_1,
\dots, Z_m\} = \mathfrak{z}$.  We can compute
\begin{equation*}
\begin{split}
  (X_i f)(x,z) &= \frac{d}{dt}|_{t=0} f((x,z)\circ(t e_i, 0)) \\
  &= \frac{d}{dt}|_{t=0} f(x+t e_i, z+\frac{1}{2}t[x, e_i]) \\
  &= \frac{d}{dt}|_{t=0} f(x+t e_i, z+\frac{1}{2}t\sum_{j} \inner{J_{u_j} x}{e_i}
  u_j) \\
  &= \left(\frac{\partial}{\partial x^i} + \frac{1}{2}\sum_j \inner{J_{u_j} x}{e_i}
  \frac{\partial}{\partial z^j}\right) f
\end{split}
\end{equation*}
So we have
\begin{align}
X_i &= \frac{\partial}{\partial x^i} + \frac{1}{2}\sum_j \inner{J_{u_j}
  x}{e_i} \frac{\partial}{\partial z^j} \label{Xi-coords} \\
Z_j &= \frac{\partial}{\partial z^j}. \label{Zj-coords}
\end{align}
The \emph{(sub-)gradient} on $G$ is given in these coordinates by 
\begin{equation}\label{gradient}
\grad f(x,z) = \sum e_i X_i f(x,z)  = \grad_x f(x,z) +
  \frac{1}{2}J_{\grad_z f(x,z)} x.
\end{equation}
Note in particular, if $f$ is radial, so that $f(x,z) = f(\abs{x}, \abs{z})$, this becomes
\begin{equation}\label{gradient-radial}
  \grad f(x,z) = f_{\abs{x}} (\abs{x}, \abs{z}) \hat{x} + \frac{1}{2}f_{\abs{z}}(\abs{x},
  \abs{z}) \abs{x} J_{\hat{z}} \hat{x} 
\end{equation}
where we use the notation $\hat{u} := \frac{u}{\abs{u}}$ to denote the
unit vector in the $u$ direction.  We draw attention to the fact that
$\hat{x}$ and $J_{\hat{z}}\hat{x}$ are orthogonal unit vectors in
$\R^{2n}$ for any nonzero $x,z$.

\section{Subriemannian geometry}

Our desired estimate for the heat kernel $p_t$ is in terms of the
Carnot-Carath\'eodory distance $d$, which is best described in the
language of subriemannian geometry.  The goal of this section will be
to obtain an explicit formula for $d$, and along the way we record
formula for the geodesics of $G$.  The computation is a
straightforward application of Hamiltonian mechanics, but we have not
seen it appear in the literature in the case of H-type groups.  The
corresponding computation for the Heisenberg groups (where the center
has dimension $m=1$) appeared in \cite{bgg} as well as
\cite{calin-book}; a computation for $m \le 7$, which could be
extended without great difficulty, can be found in the preprint
\cite{calin-H-type}.

\begin{definition}
  A \emph{subriemmanian manifold} is a smooth manifold $Q$ together
  with a subbundle $\mathcal{H}$ of $TQ$ (the \emph{horizontal
    bundle} or \emph{horizontal distribution}, whose elements are \emph{horizontal vectors}) and a
  metric $\inner{\cdot}{\cdot}_q$ on each fiber $\mathcal{H}_q$,
  depending smoothly on $q \in Q$.  $\mathcal{H}$ is
  \emph{bracket-generating} at $q$ if there is a local frame $\{X_i\}$
  for $\mathcal{H}$ near $q$ such that $\spanop\{X_i(q), [X_i,
    X_j](q), [X_i, [X_j, X_k]](q), \dots\} = T_q Q$.
\end{definition}

An H-type group $G$ can naturally be equipped as a subriemannian
manifold, by letting $\mathcal{H}_g := \{X(g) : X \in
\mathfrak{z}^\perp\}$, and using the inner product on $\mathfrak{g}$
as the metric on $\mathcal{H}$.  In other words, $\mathcal{H}_g$ is
spanned by $\{X_1(g), \dots, X_{2n}(g)\}$, which give it an
orthonormal basis.  The bracket generating condition is
obviously satisfied, since $\mathfrak{g} = \mathfrak{z}^\perp \oplus
[\mathfrak{z}^\perp,\mathfrak{z}^\perp]$.

\begin{definition}
  Let $\gamma : [0,1] \to Q$ be an absolutely continuous path.  We say
  $\gamma$ is \emph{horizontal} if $\dot{\gamma}(t) \in H_{\gamma(t)}$
  for almost every $t \in [0,1]$.  In such a case we define the
  \emph{length} of $\gamma$ as $\ell(\gamma) := \int_0^1
  \sqrt{\inner{\dot{\gamma}(t)}{\dot{\gamma}(t)}_{\gamma(t)}} \,dt$.
  The \emph{Carnot-Carath\'eodory distance} $d : Q \times Q \to [0,\infty]$ is
  defined by
  \begin{equation}
    d(q_1,q_2) = \inf\{\ell(\gamma) : \gamma(0)=q_1, \gamma(1)=q_2, \gamma
    \text{ horizontal}\}.
  \end{equation}
\end{definition}

Under the bracket generating condition, the Carnot-Carath\'eodory
distance is well behaved.  We refer the reader to Chapter 2 and
Appendix D of \cite{montgomery} for proofs of the following two theorems.

\begin{theorem}[Chow]
  If $\mathcal{H}$ is bracket generating and $Q$ is connected, then
  any two points $q_1, q_2 \in Q$ are joined by a horizontal path whose length
  is finite.  Thus $d(q_1, q_2) < \infty$, and $d$ is easily seen to
  be a distance function on $Q$.  The topology induced by $d$ is equal
  to the manifold topology for $Q$.
\end{theorem}

\begin{theorem}
  If $Q$ is complete under the Carnot-Carath\'eodory distance $d$,
  then the infimum in the definition of $d$ is achieved; that is, any
  two points $q_1, q_2 \in Q$ are joined by at least one shortest
  horizontal path.
\end{theorem}

For an H-type group, we obtain the following explicit formula for the
distance.  Note that by its definition, $d$ is left-invariant,
i.e. $d(g,h) = d(kg, kh)$, so it is sufficient to compute distance
from the identity.  By an abuse of notation, we write $d(x,z)$ to mean
$d((0,0),(x,z))$.  

\begin{theorem}\label{distance}
  Define the function $\nu : \R \to \R$ by
  \begin{equation} \label{nu-def}
  \nu(\theta) = \frac{2\theta - \sin 2 \theta}{1-\cos 2 \theta} 
  = 
  \frac{\theta}{\sin^2 \theta} - \cot\theta = -\frac{d}{d\theta} [\theta
  \cot \theta]
  \end{equation}
  where the alternate form comes from the double-angle identities.
  Then
  \begin{equation}\label{distance-formula}
    d(x,z) = 
    \begin{cases}
      \abs{x} \frac{\theta}{\sin\theta}, &  z \ne 0, x \ne 0 \\
      \abs{x}, & z=0 \\
      \sqrt{4\pi\abs{z}}, & x=0
    \end{cases}
  \end{equation}
  where $\theta$ is the unique solution in $[0,\pi)$ to $\nu(\theta) =
    \frac{4 \abs{z}}{\abs{x}^2}$.
\end{theorem}

We note that it is apparent from (\ref{distance-formula}) that we have
the scaling property
\begin{equation}\label{distance-dilate}
  d(\varphi_\alpha(x,z)) = \alpha d(x,z)
\end{equation}
with $\varphi$ as in (\ref{dilation}).

One way to compute the Carnot-Carath\'eodory distance is to find such
a shortest path and compute its length.  To find a shortest path, we
use Hamiltonian mechanics, following Chapters 1 and 5 of
\cite{montgomery}.  Roughly speaking, it can be shown that a length
minimizing path also minimizes the energy $\frac{1}{2}\int_0^1
\norm{\dot{\gamma}(t)}^2\,dt$, and as such should solve Hamilton's
equations of motion.  The argument uses the method of Lagrange
multipliers, and requires that the endpoint map taking horizontal
paths to their endpoints has a surjective differential.  This always
holds in the Riemannian setting, but is not generally true in
subriemannian geometry; the Martinet distribution (see Chapter 3 of
\cite{montgomery}) is a counterexample in which some shortest paths do
not satisfy Hamilton's equations.  Additional assumptions on
$\mathcal{H}$ are needed.  One which is sufficient (but certainly not
necessary) is that the distribution be \emph{fat}:

\begin{definition}\label{fat-def}
  Let $\Theta$ be the canonical $1$-form on the cotangent bundle
  $T^*Q$, $\omega = d\Theta$ the canonical symplectic $2$-form, and
  let $\mathcal{H}^0 := \{ p_q \in T^*Q : p_q(\mathcal{H}_q) = 0\}$ be
  the annihilator of $\mathcal{H}$.  (Note $\mathcal{H}^0$ is a
  sub-bundle, and hence also a submanifold, of $T^*Q$.)  We say
  $\mathcal{H}$ is \emph{fat} if $\mathcal{H}^0$ is symplectic away
  from the zero section.  That is, if $p_q \in \mathcal{H}^0$ is not
  in the zero section, $v \in T_{p_q}\mathcal{H}^0$, and $\omega(v,
  w)=0$ for all other $w \in T_{p_q}\mathcal{H}^0$, then $v=0$.
\end{definition}

\begin{definition}
  If $(Q, \mathcal{H}, \inner{\cdot}{\cdot})$ is a subriemannian
  manifold, the \emph{subriemannian Hamiltonian} $H : T^*Q \to \R$ is
  defined by
  \begin{equation}\label{hamiltonian-subriemannian}
    H(p_q) = \frac{1}{2}\sum_i p_q(v_i)^2
  \end{equation}
  where $\{v_i\}$ is an orthonormal basis for $(\mathcal{H}_q,
  \inner{\cdot}{\cdot}_q)$.  It is clear that this definition is
  independent of the chosen basis.  Let the \emph{Hamiltonian vector
    field} $X_H$ on $T^*Q$ be the unique vector field satisfying $dH +
  \omega(X_H, \cdot) = 0$ (as elements of $T^*T^*Q$).  $X_H$ is well
  defined because $\omega$ is symplectic.  \emph{Hamilton's equations
    of motion} are the ODEs for the integral curves of $X_H$.
\end{definition}

The following theorem summarizes (a special case of) the argument of
Chapters 1 and 5 of \cite{montgomery}.

\begin{theorem}
  If $\mathcal{H}$ is fat, then any length minimizing path $\sigma :
  [0,1] \to Q$, when parametrized with constant speed, is also energy
  minimizing and is the projection onto $Q$ of a path ${\gamma} :
  [0,1] \to T^*Q$ which satisfies Hamilton's equations of motion:
  $\dot{\gamma}(t) = X_H(\gamma(t))$.
\end{theorem}

We now verify explicitly that this theorem applies to H-type groups.
We first adopt a coordinate system for the cotangent bundle $T^*G$.

\begin{notation}
  Let $(x, z, \xi, \eta) : T^*G \to \R^{2n} \times \R^{m} \times
  \R^{2n} \times \R^m$ be the coordinate system on $T^*G$ such that
  $x^i(p_g) = x^i(g)$, $z^j(p_g)=z^j(g)$, $\xi_i(p_g) =
  p(\frac{\partial}{\partial x^i})$, $\eta_j(p_g) =
  p(\frac{\partial}{\partial z^j})$.  That is,
  \begin{equation*}
    p_g = \left(x(g), z(g),
   \sum_i \xi_i dx^i + \sum_j \eta_j dz^j\right).
  \end{equation*}
  In these
  coordinates, the canonical $2$-form $\omega$ has the expression
  $\omega = \sum_i  d\xi_i \wedge dx^i + \sum_j d\eta_j \wedge dz^j$.
\end{notation}

\begin{proposition}
If $G$ is an H-type group with horizontal distribution $\mathcal{H}$
spanned by the vector fields $X_i$, then $\mathcal{H}$ is fat.
\end{proposition}

\begin{proof}
For an H-type group $G$, we have $p_g \in \mathcal{H}^0$ iff
$p_g(X_i(g)) = 0$ for all $i$.  We can thus form a basis for
$\mathcal{H}^0_g \subset T^*_g G$ by
\begin{align*}
  w^j &= dz^j - \sum_i dz^j(X_i(g)) dx^i \\
  &= dz^j - \frac{1}{2}\sum_i \innerp{J_{u_j} x(g)}{e_i} dx^i.
\end{align*}
Expressing $p_g$ in this basis as $p_g = \sum_j \theta_j w^j$ yields a
system of coordinates $(x, z, \theta)$ for $\mathcal{H}^0$, where
$\theta$ can be identified with the element $(\theta^1, \dots,
\theta^m)$ of $\R^m$.  In terms of the coordinates $(x,z,\xi,\eta)$
for $T^*G$, we have $\eta = \theta$, $\xi = -\frac{1}{2}J_\theta x$.

So let $\gamma : (-\epsilon, \epsilon) \to \mathcal{H}^0$ be a curve
in $\mathcal{H}^0$ which avoids the zero section.  $\dot{\gamma}(0)$
is thus a generic element of $T \mathcal{H}^0$.  We write
$\gamma(t)$ in coordinates as $(x(t), z(t), \theta(t))$, where
$\theta(t) \ne 0$. In terms of the coordinates $(x,z,\xi,\eta)$ on
$T^*G$, we have $\eta(t)=\theta(t), \xi(t) = -\frac{1}{2}J_{\theta(t)}x(t)$.
Differentiating the latter gives 
\begin{equation*}
\dot{\xi}(t) =
-\frac{1}{2} (J_{\dot{\theta}(t)}x(t) + J_{\theta(t)}\dot{x}(t)).
\end{equation*}

Suppose that for all other such curves $\gamma'$ with $\gamma'(0)=\gamma(0)$, we have
$\omega(\dot{\gamma}(0), \dot{\gamma}'(0)) = 0$.  In terms of
coordinates,
\begin{align*}
  0 = \omega(\dot{\gamma}(0), \dot{\gamma}'(0)) 
  &= \sum_i (\dot{\xi}_i(0) \dot{x}'^i(0) -
  \dot{\xi}'_i(0)\dot{x}^i(0)) + \sum_j (\dot{\eta}_j(0)\dot{z}'^j(0)
  - \dot{\eta}'_j(0)\dot{z}^j(0)) \\
  &= \inner{\dot{\xi}(0)}{\dot{x}'(0)} -
  \inner{\dot{\xi}'(0)}{\dot{x}(0)} +
  \inner{\dot{\eta}(0)}{\dot{z}'(0)} -
  \inner{\dot{\eta}'(0)}{\dot{z}(0)} \\
  &= -\frac{1}{2} \inner{J_{\dot{\theta}(0)}x(0) +
    J_{\theta(0)}\dot{x}(0)}{\dot{x}'(0)} +\frac{1}{2} 
  \inner{J_{\dot{\theta}'(0)}x'(0) +
    J_{\theta'(0)}\dot{x}'(0)}{\dot{x}(0)} \\
  &\quad +
  \inner{\dot{\theta}(0)}{\dot{z}'(0)} -
  \inner{\dot{\theta}'(0)}{\dot{z}(0)} \\
  &= \frac{1}{2} \inner{x(0)}{J_{\dot{\theta}(0)} \dot{x}'(0) +
    J_{\dot{\theta}'(0)} \dot{x}(0)} + \inner{J_{\theta(0)} \dot{x}'(0)}{\dot{x}(0)} \\
  &\quad +
  \inner{\dot{\theta}(0)}{\dot{z}'(0)} -
  \inner{\dot{\theta}'(0)}{\dot{z}(0)}
\end{align*}

For arbitrary $u \in \R^m$, take $\gamma'(t) = (x(0), z(0)+tu,
\theta(0))$; then $0 = \omega(\dot{\gamma}(0), \dot{\gamma}'(0)) =
\inner{\dot{\theta}(0)}{u}$, so we must have $\dot{\theta}(0)=0$.
Next, for arbitrary $v \in \R^{2n}$, take $\gamma'(t) = (x(0) + tv,
z(0), \theta(0))$; then we have $0 =
\inner{J_{\theta(0)}u}{\dot{x}(0)}$.  But $\theta(0)\ne 0$ by
assumption, so $J_{\theta_0}$ is nonsingular and we must have
$\dot{x}(0)=0$.  Finally, take $\gamma'(t) = (x(0), z(0),
\theta(0)+tu)$; then $\inner{u}{\dot{z}(0)}=0$, so $\dot{z}(0)=0$.
Thus we have shown that if $\omega(\dot{\gamma}(0), \dot{\gamma}'(0))
= 0$ for all $\gamma'$, we must have $\dot{\gamma}(0)=0$, which
completes the proof.
\end{proof}

We now proceed to compute and solve Hamilton's equations of motion for
an H-type group $G$.

The subriemannian Hamiltonian on $T^*G$ is defined by (c.f. (\ref{hamiltonian-subriemannian}))
\begin{equation}\label{hamiltonian-h-type}
  H(p_g) := \frac{1}{2}\sum_{i=1}^{2n} p_g(X_i(g))^2, \quad p_g \in T^*_g G.
\end{equation}

In terms of the above coordinates, we may compute
\begin{equation*}
  p_g(X_i(g)) = p_g\left(\frac{\partial}{\partial
    x^i} + \frac{1}{2} \sum_j \inner{J_{u_j} x}{e_i}
  \frac{\partial}{\partial z^j}\right) = \xi_i(p_g) + \frac{1}{2}\inner{J_{\eta(p_g)}
    x(g)}{e_i}
\end{equation*}
so that
\begin{equation*}
  H(p_g) = \frac{1}{2}\abs{\xi(p_g) + \frac{1}{2} J_{\eta(p_g)} x(g)}^2.
\end{equation*}

Recall that a path $\gamma : [0,T] \to T^*Q$ satisfies
Hamilton's equations iff $\dot{\gamma}(t) =
X_H(\gamma(t))$, i.e. $dH_{\gamma(t)} + \omega(\dot{\gamma}(t),
\cdot) = 0$.

In an $H$-type group $G$, we write $\gamma$ in coordinates as
$\gamma(t) = (x(t), z(t), \xi(t), \eta(t)) : [0,T] \to T^*G$, so that
we have
\begin{align*}
  \omega(\dot{\gamma}(t), \cdot) = \sum_i (\dot{\xi}_i(t) dx^i -
  \dot{x}^i(t) d\xi_i) + \sum_j (\dot{\eta}_j(t) dz^j - \dot{z}^j(t) d\eta_j).
\end{align*}
Thus Hamilton's equations of motion read
\begin{equation}
  \dot{x}^i = \frac{\partial H}{\partial \xi^i}, \quad
  \dot{\xi}_i = -\frac{\partial H}{\partial x^i}, \quad
  \dot{z}_j = \frac{\partial H}{\partial \eta^j},\quad
  \dot{\eta}_j = -\frac{\partial H}{\partial z^j}.
\end{equation}

To compute the derivatives, we note that $\frac{1}{2} \grad_x
\abs{Ax+y}^2 = A^* Ax + A^*y$.  If we write $B_x \eta =
J_\eta x$, then $\inner{B_x \eta}{y} = \inner{\eta}{[x,y]}$,
so $B_x^* = [x,\cdot]$, and $B_x^* B_x = \abs{x}^2 I$.  So for a
path $\gamma(t) = (x(t), z(t), \xi(t), \eta(t)) : [0,T] \to T^*G$,
Hamilton's equations of motion read
\begin{align}
  \dot{x} &= \grad_\xi H = \xi + \frac{1}{2} J_\eta x \label{Hx} \\
  \dot{z} &= \grad_\eta H = \frac{1}{2} \grad_\eta \abs{\xi + \frac{1}{2} B_x \eta}^2 =
  \frac{1}{4} \abs{x}^2 \eta + \frac{1}{2}[x,\xi] \label{Hz} \\
  \dot{\xi} &= -\grad_x H = -\frac{1}{4}\abs{\eta}^2 x + \frac{1}{2}J_\eta \xi \label{Hxi}\\ 
  \dot{\eta} &= -\grad_z H = 0. \label{Heta}
\end{align}

\begin{theorem}\label{ham-soln}
  $(x(t), z(t))$ is the projection of a solution to Hamilton's
  equations with $x(0)=0, z(0)=0$ and $x(1), z(1)$ given if and only if:
  \begin{enumerate}
  \item If $z(1)=0$, we have
    \begin{equation}\label{ham-line}
      x(t) = tx(1),\quad z(t)=0.
    \end{equation}
  \item If $z(1) \ne 0$, we have
\begin{align}
  x(t) &= 
  \frac{1}{\abs{\eta_0}^2}J_{\eta_0}(I-e^{tJ_{\eta_0}})\xi_0 \label{ham-arc-x}
  \\  
  z(t) &=
  \frac{\abs{\xi_0}^2}{2\abs{\eta_0}^3}(\abs{\eta_0}t-\sin(\abs{\eta_0}t))
  \eta_0 \label{ham-arc-z}
\end{align}
where, if $x(1) \ne 0$ we have
\begin{align}
  \eta_0 &= 2\theta \frac{z(1)}{\abs{z(1)}} \\
  \xi_0 &= -\abs{\eta_0}^2 (J_{\eta_0} (e^{J_{\eta_0}}-I))^{-1} x(1).
       \intertext{where $\theta$ is a solution to}
       \nu(\theta) &= \frac{4\abs{z(1)}}{\abs{x(1)}^2}; \label{nu-dist}
\end{align}
and if $x(1) = 0$ we have
\begin{align*}
  \eta_0 &= 2 \pi k \frac{z(1)}{\abs{z(1)}} \\
  \abs{\xi_0} &= \sqrt{4 k \pi \abs{z(1)}}
\end{align*}
for some integer $k \ge 1$.
  \end{enumerate}
\end{theorem}

\begin{proof}
  
We solve (\ref{Hx}--\ref{Heta}), assuming $x(0)=0$, $z(0)=0$.  By
(\ref{Heta}) we have $\eta(t) \equiv \eta(0) = \eta_0$.  If $\eta_0 =
0$, we can see by inspection that the solution is 
\begin{equation}\label{eta0-0-soln}
\eta(t) = 0, \quad \xi(t) = \xi_0, \quad x(t)=t\xi_0, \quad z(t)=0,
\end{equation}
namely, a straight line from the origin, whose length is clearly
$\abs{x(1)}$.  This is (\ref{ham-line}), which
we shall see is forced when $z(1)=0$.

Otherwise, assume $\eta_0 \ne 0$.  We may solve (\ref{Hx}) for $\xi$
to see that
\begin{equation} \label{xi-solved}
\xi = \dot{x} - \frac{1}{2} J_{\eta_0} x.
\end{equation}  
Notice that substituting (\ref{xi-solved}) into (\ref{Hz}) shows that
\begin{equation}\label{zdot-bracket}
  \dot{z} = \frac{1}{2}[x, \dot{x}],
\end{equation}
from which an easy computation
verifies that $(x(t),z(t))$ is indeed a horizontal path.

Substituting (\ref{xi-solved}) into the right side of (\ref{Hxi}) shows that
\begin{equation*}
  \dot{\xi} = -\frac{1}{4} \abs{\eta_0}^2 x + \frac{1}{2} J_{\eta_0}
  (\dot{x} - \frac{1}{2} J_{\eta_0} x) =
  \frac{1}{2} J_{\eta_0} \dot{x}
\end{equation*}
since $J_{\eta_0}^2 x = -\abs{\eta_0}^2 x$.  Thus 
\begin{equation}\label{Hxi-x}
  \xi = \frac{1}{2} J_{\eta_0} x + \xi_0
\end{equation}
where $\xi_0 = \xi(0)$.  If $\xi_0 = 0$, it is easily seen that we
have the trivial solution $x(t)=0$, $z(t)=0$, $\xi(t)=0$, $\eta(t)=\eta_0$,
so we assume now that $\xi_0 \ne 0$.  (\ref{Hxi-x}) may be substituted back
into (\ref{Hx}) to get
\begin{equation}
  \dot{x} =  J_{\eta_0} x + \xi_0
\end{equation}
so that
\begin{equation}\label{Hxsol}
  x = (J_{\eta_0})^{-1} (e^{tJ_{\eta_0}}-I)\xi_0 =
  -\frac{1}{\abs{\eta_0}^2}J_{\eta_0}(e^{tJ_{\eta_0}}-I)\xi_0.
\end{equation}
Differentiation (or substitution) shows 
\begin{equation}\label{xdot}
\dot{x} = e^{tJ_{\eta_0}} \xi_0.
\end{equation}

Note that
\begin{equation}\label{normx-formula}
  \abs{x}^2 = \frac{1}{\abs{\eta_0}^2}
  \abs{(e^{tJ_{\eta_0}}-I)\xi_0}^2 = \frac{2}{\abs{\eta_0}^2}(1-\cos(\abs{\eta_0}t))\abs{\xi_0}^2.
\end{equation}

It is easy to see from (\ref{Hxsol}) that $x(t)$ lies in the plane
spanned by $\xi_0$ and $J_{\eta_0} \xi_0$, and $x(t)$ sweeps out a
circle centered at $\frac{1}{\abs{\eta_0}^2}J_{\eta_0} \xi_0$ and
passing through the origin.  In particular, the radius of the circle
is ${\abs{\xi_0}}/{\abs{\eta_0}}$.

Now substituting (\ref{Hxsol}) and (\ref{xdot}) into
(\ref{zdot-bracket}), we have
\begin{align*}
  \dot{z} &= -\frac{1}{2 \abs{\eta_0}^2}\left([J_{\eta_0} e^{tJ_{\eta_0}}
    \xi_0, e^{tJ_{\eta_0}} \xi_0] - [J_{\eta_0} \xi_0, e^{tJ_{\eta_0}}
    \xi_0]\right) \\
  &= \frac{1}{2 \abs{\eta_0}^2}\left(\abs{e^{tJ_{\eta_0}}
    \xi_0}^2 \eta_0 + [J_{\eta_0}\xi_0, e^{tJ_{\eta_0}} \xi_0]\right) \\
  &= \frac{\abs{\xi_0}^2}{2 \abs{\eta_0}^2}\left(1 -
  \cos(\abs{\eta_0} t)\right) \eta_0
\end{align*}

By integration,
\begin{equation}\label{z-formula}
  z = \frac{\abs{\xi_0}^2}{2\abs{\eta_0}^3}(\abs{\eta_0}t-\sin(\abs{\eta_0}t)) \eta_0.
\end{equation}
In particular,
\begin{equation}\label{normz-formula}
  \abs{z} = \frac{\abs{\xi_0}^2}{2\abs{\eta_0}^2}( \abs{\eta_0}t-\sin(\abs{\eta_0}t)).
\end{equation}
We note that inspection of (\ref{normz-formula}) shows that $z(t) \ne
0$ for $t > 0$.  Thus the only solution with $z(1)=0$ is that of
(\ref{ham-line}). 

To make more sense of this, let $r=\abs{\xi_0}/\abs{\eta_0}$ be the
radius of the arc swept out by $x(t)$, and $\phi =  \abs{\eta_0} t$
be the angle subtended by the arc.  Then
\begin{equation*}
  \abs{z} = \frac{1}{2} r^2 \phi - \frac{1}{2} r^2 \sin \phi
\end{equation*}
which is the area of the region between an arc of radius $r$
subtending an angle $\phi$ and the chord which spans it.

We must determine $\xi_0, \eta_0$ in terms of $x(1), z(1)$.  We have
already ruled out the case $z(1)=0$.  If $x(1)=0$, then
(\ref{normx-formula}) shows we must have $\abs{\eta_0} = 2 k \pi$ for
some integer $k \ge 1$.  (\ref{z-formula}, \ref{normz-formula}) then
shows $\eta_0 = 2k \pi z(1) / \abs{z(1)}$, and $\abs{\xi_0} =
\sqrt{4k\pi\abs{z(1)}}$, as desired.  In this case the direction of
$\xi_0$ is not determined and $\xi_0$ may be any vector with the given
length.

On the other hand, if $x(1) \ne 0$, then $\abs{\eta_0}$ is not an
integer multiple of $2\pi$, so we may divide (\ref{normz-formula}) by
(\ref{normx-formula}) to obtain
\begin{equation}\label{eta0-eqn}
  \frac{\abs{z(1)}}{\abs{x(1)}^2} = \frac{\abs{\eta_0} - \sin
    \abs{\eta_0}}{4(1-\cos \abs{\eta_0})} = \frac{1}{4}
  \nu(\theta)
\end{equation}
taking $\theta = \frac{1}{2} \abs{\eta_0}$, where $\nu$ is as in (\ref{nu-def}).
Then by (\ref{normx-formula}) we have
\begin{equation}\label{norm-xi0}
  \abs{\xi_0}^2 = \frac{1}{2} \abs{x(1)}^2 \frac{
    \abs{\eta_0}^2}{1-\cos(\abs{\eta_0})} = \abs{x(1)}^2
  \frac{\theta^2}{\sin^2\theta}.
\end{equation}

Note that once the magnitudes of $\eta_0$, $\xi_0$ are known, their
directions are determined: $\eta_0 = z(1) \abs{\eta_0} / \abs{z(1)}$
by (\ref{z-formula}), while $\xi_0$ can be recovered from (\ref{Hxsol}):
\begin{align*}
  \xi_0 &= -\eta_0^2 (J_{\eta_0} (e^{J_{\eta_0}}-I))^{-1} x(1).
\end{align*}
So $\eta_0, \xi_0$ and hence $x(t), z(t)$ are all determined by a
choice of $\abs{\eta_0}$ satisfying (\ref{eta0-eqn}).  Writing $\theta
= \abs{\eta_0}$ gives (\ref{ham-arc-x}--\ref{ham-arc-z}).

The ``if'' direction of the theorem requires verifying that the given
formulas in fact satisfy Hamilton's equations, which is routine.
\end{proof}


To prove Theorem \ref{distance}, we must now decide which of the
solutions given in Theorem \ref{ham-soln} is the shortest, and compute
its length.  We collect, for future reference, some facts about the
function $\nu$ of (\ref{nu-def}).

\begin{lemma}\label{nu-increase}
  There is a constant $c > 0$ such that $\nu'(\theta) > c$ for all
  $\theta \in [0,\pi)$.
\end{lemma}

\begin{proof}
  By direct computation, $\nu'(\theta) =
  \frac{2(\sin\theta-\theta\cos\theta)}{\sin^3\theta}$.  By Taylor
  expansion of the numerator and denominator we have $\nu'(0) = 2/3 >
  0$.  For all $\theta \in (0,\pi)$ we have $\sin^3\theta > 0$, so it
  suffices to consider $y(\theta) := \sin\theta-\theta\cos\theta$.
  Now $y(0)=0$ and $y'(\theta) = \theta\sin\theta > 0$ for $\theta \in
  (0,\pi)$, so $y(\theta) > 0$ for $\theta \in (0,\pi)$.  Thus
  $\nu'(\theta) > 0$ for $\theta \in [0,\pi)$, and continuity and the
    fact that $\lim_{\theta \uparrow \pi} \nu'(\theta) = +\infty$
    establishes the existence of the constant $c$.
\end{proof}

\begin{corollary}\label{nu-c}
  $\nu(\theta) \ge c\theta$ for all
  $\theta \in [0,\pi)$, where $c$ is the constant from Lemma \ref{nu-increase}.
\end{corollary}

\begin{proof}
  Integrate the inequality in Lemma \ref{nu-increase}.  Note that $\nu(0)=0$.
\end{proof}

\begin{proof}[Proof of Theorem \ref{distance}]
  We compute the lengths of the paths given in Lemma \ref{ham-soln}.
  The $z=0$ case is obvious.  Observe that for a horizontal path
  $\sigma(t) = (x(t), z(t))$, we have $\dot{\sigma}(t) =
  \sum_{i=1}^{2n} \dot{x}^i(t) X_i(\gamma(t))$, so that
  $\norm{\dot{\sigma}(t)} = \abs{\dot{x}(t)}$.  For paths solving
  Hamilton's equations, (\ref{xdot}) shows that $\abs{\dot{x}(t)} =
  \abs{\xi_0}$, so $\ell(\gamma)=\abs{\xi_0}$.  In the case $x=0$, we
  have $\abs{\xi_0} = \sqrt{4k\pi\abs{z(1)}}$, where $k$ may be any
  positive integer; clearly this is minimized by taking $k=1$.

  Now we must handle the case $x \ne 0$, $z \ne 0$.  In this case we
  have $\ell(\gamma) = \abs{\xi_0} = \abs{x}
  \frac{\theta}{\sin\theta}$, by (\ref{norm-xi0}), where $\theta$
  solves (\ref{nu-dist}) (recall $\theta = \frac{1}{2} \abs{\eta_0}$).  The
  function $\nu$ has $\nu(0)=0$, $\nu(\pi)=+\infty$, and by Lemma
  \ref{nu-increase} $\nu$ is strictly increasing on $[0,\pi)$.
    \ignorethis{(]} Thus among the solutions of (\ref{nu-dist}) there
  is exactly one in $[0,\pi)$.  \ignorethis{(]} We show this is the
  solution that minimizes $\left(\frac{\theta}{\sin\theta}\right)^2$
  and hence also minimizes $\ell(\gamma)$.

  For brevity, let $y=\frac{4 \abs{z}}{\abs{x}^2}$.  If $y \in
  [0,\pi/2]$ then $y=\nu(\theta)$ for a unique $\theta \in
  [0,\infty)$.  \ignorethis{(]} This is because $\nu(\theta) >
  \nu(\pi/2) = \pi/2$ for $\theta > \pi/2$.  Since $\theta$ is
  increasing on $[0,\pi)$ \ignorethis{(]} it suffices to show this for
  $\theta > \pi$.  But for such $\theta$ we have
  \begin{align*}
    \nu(\theta) = \frac{\theta - \sin\theta \cos\theta}{\sin^2\theta}
    \ge \frac{\theta - \frac{1}{2}}{\sin^2\theta} \ge \theta -
    \frac{1}{2} > \pi-\frac{1}{2} > \frac{\pi}{2}
  \end{align*}
  since $\sin\theta\cos\theta \le \frac{1}{2}$ for all $\theta$.
  
  Otherwise, suppose $y > \pi/2$.  Let
  \begin{equation*}
    F(\theta) :=
    \frac{\left(\frac{\theta}{\sin\theta}\right)^2}{\nu(\theta)} =
    \frac{\theta^2}{\theta - \sin\theta\cos\theta}
  \end{equation*}
  which is smooth on $(\pi/2,\infty)$ after removing the removable
  singularities.  We will show that if $\pi/2 < \theta_1 < \pi <
  \theta_2$, then $F(\theta_1) < F(\theta_2)$.  Thus if $\theta_1$ is
  the unique solution to $y=\nu(\theta)$ in $(\pi/2,\pi)$ and
  $\theta_2 > \pi$ is another solution, we will have
  \begin{equation*}
    \left(\frac{\theta_1}{\sin\theta_1}\right)^2 = \nu(\theta_1)
    F(\theta_1) = y F(\theta_1) < y F(\theta_2) = \nu(\theta_2)
    F(\theta_2) = \left(\frac{\theta_2}{\sin\theta_2}\right)^2
  \end{equation*}

Toward this end, we compute
  \begin{align*}
    F'(\theta) &= \frac{2\theta(\theta-\sin\theta\cos\theta) -
      \theta^2(1-\cos^2\theta+\sin^2\theta)}{(\theta-\sin\theta\cos\theta)^2}
    \\
    &= \frac{2\theta\cos\theta(\theta\cos\theta-\sin\theta)}{(\theta-\sin\theta\cos\theta)^2}.
  \end{align*}
  For $\theta \in (\pi/2,\pi)$ we have $\cos\theta < 0$, $\sin\theta >
  0$ and thus $F'(\theta) > 0$.  So $F(\theta_1) < F(\pi)$ and it
  suffices to show $F(\pi) = \pi < F(\theta_2)$.  We have $F'(\pi)=2>0$ so
  this is true for $\theta_2$ near $\pi$, and $F(+\infty)=+\infty$ so
  it is also true for large $\theta_2$.  To complete the argument we
  show that it holds at critical points of $F$.  Suppose
  $F'(\theta_c)=0$ where $\theta_c > \pi$; then either $\cos\theta_c=0$ or
  $\theta_c\cos\theta_c-\sin\theta_c=0$.  If the former then
  $F(\theta_c)=\theta_c>\pi$.  If the latter, then $\theta_c =
  \tan\theta_c$, so
  \begin{align*}
    F(\theta_c) = 
    \frac{\theta_c^2}{\theta_c - \sin\theta_c\cos\theta_c} =
    \frac{\theta_c^2}{\theta_c - \tan\theta_c \cos^2\theta_c} =
    \frac{\theta_c^2}{\theta_c(1- \cos^2\theta_c)} \ge \theta_c > \pi
  \end{align*}
which completes the proof.      
\end{proof}

\begin{notation}
  If $f,h : G \to \R$, we write $f(g) \asymp
h(g)$ to mean there exist finite positive constants $C_1, C_2$ such
that $C_1 h(g) \le f(g) \le C_2 h(g)$ for all $g \in G$, or some
specified subset thereof.
\end{notation}

\begin{corollary}\label{distance-estimate}
  $d(x,z) \asymp \abs{x} + \abs{z}^{1/2}$.  Equivalently, $d(x,z)^2
  \asymp \abs{x}^2 + \abs{z}$.
\end{corollary}

\begin{proof}
  By continuity we can assume $x \ne 0$, $z \ne 0$.  If $\theta$ is
  the unique solution in $[0,\pi)$ to
    $\nu(\theta)=\frac{4\abs{z}}{\abs{x}^2}$, we have $d(x,z)^2 =
    \abs{x}^2 \left(\frac{\theta}{\sin\theta}\right)^2$, so if we let
    \begin{equation}
      F(\theta) :=
      \frac{\left(\frac{\theta}{\sin\theta}\right)^2}{1+\nu(\theta)} =
      \frac{d(x,z)^2}{\abs{x}^2 + 4\abs{z}}
    \end{equation}
    it will be enough to show there exist $D_1, D_2$ with $0 < D_1 \le
    F(\theta) \le D_2$ for all $\theta \in [0,\pi)$.  $F$ is obviously
      continuous and positive on $(0,\pi)$.  We can simplify $F$ as
      \begin{equation*}
        F(\theta) = \frac{\theta^2}{\sin^2\theta + \theta - \sin\theta
        \cos\theta}
      \end{equation*}
      from which it is obvious that $\lim_{\theta \uparrow \pi}
      F(\theta) = \pi > 0$, and easy to compute that $\lim_{\theta
        \downarrow 0} F(\theta) = 1 > 0$, which is sufficient to
      establish the corollary.

\end{proof}

Results of this form apply to general stratified Lie groups.  A
standard argument, paraphrased from \cite{blu-book}, where many more
details can be found, is as follows.  Once it is known that $d$
generates the Euclidean topology on $G$, then $d(x,z)$ is a continuous
function which is positive except at $(0,0)$.  $d'(x,z) := \abs{x} +
\abs{z}^{1/2}$ is another such function, so the conclusion obviously
holds on the unit sphere of $d'$.  Now $d'(\varphi_\alpha(x,z)) =
\alpha d'(x,z)$, and inspection of (\ref{distance-formula}) shows that
the same holds for $d$, so for general $(x,z)$ it suffices to apply
the previous statement with $\alpha = d'(x,z)^{-1}$.

\section{The sublaplacian and heat kernel estimates}\label{sublaplacian-sec}

\begin{definition}
  The \emph{sublaplacian} $L$ for $G$ is
  the operator given by
  \begin{equation}\label{L-def}
    L = \sum_{i} X_i^2
  \end{equation}
  where $X_i$ are as given in (\ref{Xi-coords}).  The \emph{heat kernel}
  $p_t$ for $G$ is the unique fundamental solution to the
  corresponding heat equation $(L-\firstder{t})u = 0$; that is, $p_t =
  e^{tL}\delta_0$, where $\delta_0$ is the Dirac delta distribution
  supported at $0$.
\end{definition}

$L$ is obviously left-invariant.  $L$ is not strictly elliptic at any
point of $G$, but it is subelliptic everywhere.

If we view the left-invariant vector fields $\{X_i\}$ as elements of
the Lie algebra $\mathfrak{g}$ of $G$, they are an orthonormal basis
for $\mathfrak{z}^\perp$, which generates $\mathfrak{g}$: that is,
$\spanop\{X_i, [X_j, X_k] : i,j,k = 1,\dots, 2n\} = \mathfrak{g}$.
(It is easy to see that $L$ does not actually depend on the choice of
orthonormal basis $\{X_i\}$ for $\mathfrak{z}^\perp$, but only on the
inner product $\inner{\cdot}{\cdot}$ on $\mathfrak{g}$.) We thus have
$\spanop\{X_i(g), [X_j, X_k](g) : i,j,k = 1,\dots, 2n\} = T_g G$ for
each $g \in G$ (it is obvious for $g=0$, and for other $g$ it follows
by left invariance).  Thus the collection of vector fields $\{X_i\}$
is \emph{bracket generating}.  By a famous theorem of H\"ormander
(\cite{hormander67}), $L$ is hypoelliptic; that is, if $Lu$ is
$C^\infty$ on some open set, then so is $u$.  Another case of
H\"ormander's theorem applies to the operator $L -
\frac{\partial}{\partial t}$ on $G \times (0, \infty) = \{(g,t)\}$;
thus, since $(L - \frac{\partial}{\partial t}) p_t = 0$ is $C^\infty$,
$p_t$ itself is $C^\infty$ on $G \times (0, \infty)$.

Our next step is to record an explicit formula for $p_t(x,z)$.
Various derivations of this formula appear in the literature.  For
general step 2 nilpotent groups, \cite{gaveau77} derived such a formula
probabilistically from a formula in \cite{levy50} regarding the L\'evy
area process.  Another common approach, worked out in \cite{cygan}, involves expressing $p_t$ as the
Fourier transform of the Mehler kernel.  \cite{taylor} has a similar
computation.  \cite{randall} obtains the formula for H-type groups as
the Radon transform of the heat kernel for the Heisenberg group.
Other approaches have involved complex Hamiltonian mechanics
(\cite{bgg}), magnetic field heat kernels (\cite{klingler}), and
approximation of Brownian motion by random walks (\cite{hulanicki}).
In our notation, we find that
\begin{equation}\label{Rm-integral}
    p_t(x,z) = (2\pi)^{-m}(4\pi)^{-n} \int_{\R^m} e^{i\inner{\lambda}{z}-\frac{1}{4}\abs{\lambda} \coth(t \abs{\lambda})
    \abs{x}^2} \left(\frac{\abs{\lambda}}{\sinh( t
      \abs{\lambda})}\right)^{n} \,d\lambda.
\end{equation}
We can see directly by making the change of variables $\lambda =
\alpha^2 \lambda'$ (among other means) that
\begin{equation}\label{pt-dilate}
  p_t(x,z) = \alpha^{2(m+n)}p_{\alpha^2t}(\alpha x,\alpha^2 z)
  =\alpha^{2(m+n)}p_{\alpha^2t}(\varphi_\alpha(x,z)).
\end{equation}
In particular, taking $\alpha = t^{-1/2}$,
\begin{equation}\label{p1-dilate}
  p_t(x,z) ) = t^{-m-n}
  p_1(t^{-1/2}x, t^{-1}z) = t^{-m-n} p_1(\varphi_{t^{-1/2}}(x,z)).
\end{equation}
Therefore an estimate on $p_1$ will immediately give an estimate on
$p_t$ for all $t$, and we study $p_1$ from this point onward.

We immediately note that the integrand in (\ref{Rm-integral}) has even
real part and odd imaginary part, so that $p_1$ is indeed real.
Moreover, being the Fourier transform of a radial function, $p_1$ is
radial, i.e. $p_1(x,z)$ depends only on $\abs{x}$ and $\abs{z}$.
  So we
can apply (\ref{gradient-radial}) and differentiate under the integral
sign to get
\begin{align}
  \grad p_1(x,z) &= -\frac{1}{2}(2\pi)^{-m}(4\pi)^{-n}\abs{x} (q_1(x,z) \hat{x} + q_2(x,z)
  J_{\hat{z}}\hat{x}) \label{grad-q1-q2}\\
\intertext{where}
q_1(x,z) &= -\frac{2}{\abs{x}} \frac{\partial p_1(x,z)}{\partial
  \abs{x}} = \int_{\R^m} e^{{i} \inner{\lambda}{z} - \frac{1}{4} \abs{\lambda}\coth{\abs{\lambda}} \abs{x}^2}
\left(\frac{\abs{\lambda}}{\sinh(\abs{\lambda})}\right)^{n+1}\cosh(\abs{\lambda})\,d\lambda \label{q1-def}
\\
q_2(x,z) &=\frac{\partial p_1(x,z)}{\partial
  \abs{z}}= \int_{\R^m} e^{{i} \inner{\lambda}{z} - \frac{1}{4} \abs{\lambda}\coth{\abs{\lambda}} \abs{x}^2}
\left(\frac{\abs{\lambda}}{\sinh(\abs{\lambda})}\right)^{n} (-i)
\inner{\lambda}{\hat{z}}\,d\lambda \label{q2-def}
\end{align}

As before, (\ref{q1-def}) and (\ref{q2-def}) do not really depend on
$\hat{z}$ but only on $\abs{x},\abs{z}$.

We now state the main theorem of this paper: the precise
estimates on $p_t$ and its gradient.  The proofs will occupy the
remainder of the paper.

\begin{theorem}\label{main-theorem}
  There exists $d_0 > 0$ such that
  \begin{equation}\label{main-theorem-eqn}
  p_1(x,z) \asymp \frac{d(x,z)^{2n-m-1}}{1+(\abs{x}d(x,z))^{n-\frac{1}{2}}}
    e^{-\frac{1}{4}d(x,z)^2}.
  \end{equation}
  for $d(x,z) \ge d_0$.
\end{theorem}

\begin{corollary}\label{main-corollary}
  \begin{equation}\label{main-eqn-t}
      p_t(x,z) \asymp t^{-m-n} \frac{1 + (t^{-1/2}d(x,z))^{2n-m-1}}{1+(t^{-1}\abs{x}d(x,z))^{n-\frac{1}{2}}}
    e^{-\frac{1}{4t}d(x,z)^2}
  \end{equation}
  for $(x,z) \in G$, $t > 0$, with the implicit constants independent
  of $t$ as well as $(x,z)$.
\end{corollary}

\begin{proof}
  Theorem \ref{main-theorem} establishes (\ref{main-eqn-t}) for $t=1$
  and $d(x,z) \ge d_0$.  For $d(x,z) \le d_0$ the estimate follows
  from continuity and the fact that $p_t(x,z) > 0$.  Although the
  positivity of $p_t$ is not obvious from inspection of
  (\ref{Rm-integral}), it is well known.  A proof of this fact could
  be assembled from the fact that the semigroup $e^{tL}$ is positive
  and hence $p_t \ge 0$ (see, for instance, Theorem 5.1 of
  \cite{hunt}) together with a Harnack inequality such as Theorem
  III.2.1 of \cite{purplebook} (which is written about positive
  functions but easily extends to cover those which are nonnegative).

  Once (\ref{main-eqn-t}) holds for all $(x,z)$ and $t=1$,
  (\ref{pt-dilate}) and (\ref{distance-dilate}) show that it holds for
  all $t$, with the same constants.
\end{proof}

We also obtain precise upper and lower estimates on the gradient of
the heat kernel.  Again we work only on $d(x,z) \ge d_0$, and since
$\grad p_t$ vanishes for $x=0$, it is not as clear how to extend to
all of $G$.  However, the upper bound is sufficient to establish
(\ref{gradient-crude-upper}), which is of interest itself.

\begin{theorem}\label{main-gradient-theorem}
  There exists $d_0 > 0$ such that
  \begin{equation}\label{main-gradient-eqn}
  \abs{\grad p_1(x,z)} \asymp \abs{x}
  \frac{d(x,z)^{2n-m+1}}{1+(\abs{x}d(x,z))^{n+\frac{1}{2}}}e^{-\frac{1}{4}d(x,z)^2}
  \end{equation}
  for $d(x,z) \ge d_0$.  In particular, we can combine this with the
  lower bound of Theorem \ref{main-theorem} to see that there exists
  $C > 0$ such that
  \begin{equation}\label{gradient-crude-upper}
    \abs{\grad p_1(x,z)} \le C(1+d(x,z))p_1(x,z).
  \end{equation}
\end{theorem}

The function $q_2$ is of interest in its own right, because it gives
the norm of the ``vertical gradient'' of $p_1$: $\abs{q_2} =
\abs{(Z_1 p_1, \dots, Z_m p_1)}$.  The proof of Theorem
\ref{main-gradient-theorem} includes estimates on $q_2$; we record
here the upper bound.
\begin{theorem}\label{vertical-gradient-theorem}
  There exists $d_0 \ge 0$ and a constant $C > 0$ such that
  \begin{equation}
    \abs{(Z_1 p_1, \dots, Z_m p_1)(x,z)} = \abs{q_2(x,z)} \le C
    \frac{d(x,z)^{2n-m-1}}{1+(\abs{x}d(x,z))^{n-\frac{1}{2}}} e^{-\frac{1}{4}d(x,z)^2}.
  \end{equation}
  whenever $d(x,z) \ge d_0$.  In particular, for all $(x,z) \in G$ we have
  \begin{equation}\label{vertical-gradient-crude}
    \abs{(Z_1 p_1, \dots, Z_m p_1)(x,z)} \le C p(x,z).
  \end{equation}
  \end{theorem}

\begin{remark*}
  Since our estimate is based on analysis of the formula
  (\ref{Rm-integral}), we will henceforth treat (\ref{Rm-integral}) as
  the definition of a function $p_1$ on $\R^{2n+m}$.  In particular,
  it makes sense for all $n,m$, whether or not an H-type group of the
  corresponding dimension actually exists (which can be ascertained
  via Theorem \ref{dimension-classification}).  The proofs of Theorems
  \ref{main-theorem} and \ref{main-gradient-theorem} do not depend on
  the values of $n$ and $m$, so they likewise remain valid for all
  $n,m$.  The estimates given are in terms of the distance function
  $d$, which likewise should be taken as a function defined by the
  formula (\ref{distance-formula}).  Indeed, the only place where we need
  $p_1$ to be a heat kernel is in the proof of Corollary
  \ref{main-corollary}, where we use the positivity of $p_1$ which
  follows from the general theory.

  In particular, in Section \ref{hadamard-sec} we shall make use of
  estimates on $p_1$ for values of $n,m$ not necessarily corresponding
  to H-type groups.
\end{remark*}

The proofs of these two theorems are broken into two cases, depending
on the relative sizes of $\abs{x}$ and $\abs{z}$.
Section \ref{steepest-descent-sec} deals with the case when $\abs{z}
\lesssim \abs{x}^2$; here we apply a steepest descent type argument to
approximate the desired function by a Gaussian.  Section
\ref{polar-sec} handles the case $\abs{z} \gg \abs{x}^2$ by a
transformation to polar coordinates and a residue computation which
only works for odd $m$.  The result for $m$ even can be deduced from
that for $m$ odd by a Hadamard descent approach, which is contained in
Section \ref{hadamard-sec}.

\section{Steepest descent}\label{steepest-descent-sec}

We first handle the region where $\abs{z} \le B_1 \abs{x}^2$ for some
constant $B_1$.  If $\theta = \theta(x,z)$ is as in Theorem
\ref{distance}, this implies $\nu(\theta) \le 4 B_1$; since $\nu$
increases on $[0,\pi)$ \ignorethis{(]} we have $0 \le \theta \le
\theta_0$ in this region.  Note also that by Corollary
\ref{distance-estimate} we have $d(x,z)^2 \le D_2(1+B_1)\abs{x}^2$, as
well as $d(x,z)^2 \ge \abs{x}^2$ which is clear from
(\ref{distance-formula}).  Thus for this region the bounds of Theorems
\ref{main-theorem}, \ref{main-gradient-theorem} and
\ref{vertical-gradient-theorem} are implied by the following:
\begin{theorem}\label{region-I-theorem}
  For each constant $B_1 > 0$ there exists $d_0 > 0$ such that
  \begin{align}
    p_1(x,z) &\asymp \frac{1}{\abs{x}^m} e^{-\frac{1}{4} d(x,z)^2} \label{I-p-both} \\
    \abs{q_i(x,z)} &\le  \frac{C_2}{\abs{x}^{m}} e^{-\frac{1}{4}
      d(x,z)^2}, \quad i=1,2 \label{I-q-upper} \\
    \frac{C_1}{\abs{x}^{m}} e^{-\frac{1}{4} d(x,z)^2} &\le
    \max\{\abs{q_1(x,z)}, \abs{q_2(x,z)}\} \label{I-q-lower}
  \end{align}
  for all $x,z$ with $d(x,z) \ge d_0$ and $\abs{z} \le B_1 \abs{x}^2$.
\end{theorem}

Our approach here will be a steepest descent argument.  Very
informally, the motivation is as follows: given a function $F(x) =
\int_{\R} e^{-x^2 f(\lambda)} a(\lambda)\,d\lambda$, move the contour
of integration to a new contour $\Gamma$ which passes through a
critical point $\lambda_c$ of $f$, so that $f(\lambda) \approx
f(\lambda_c) + \frac{1}{2} f''(\lambda_c)(\lambda-\lambda_c)^2$.  Then
we have
\begin{equation*}
F(x) \approx
e^{-x^2 f(\lambda_c)} \int_{\Gamma} e^{-x^2 f''(\lambda_c)
  (\lambda-\lambda_c)^2/2} a(\lambda)\,d\lambda.
\end{equation*}
For large $x$ the integrand looks like a Gaussian concentrated near
$\lambda_c$, so $F(x) \asymp e^{-x^2 f(\lambda_c)}
\frac{a(\lambda_c)}{x \sqrt{f''(\lambda_c)}}$.  Our proof essentially
follows this line, in $\R^m$ instead of $\R$, but more care is
required to establish the desired uniformity.

Our first task is to extend the integrand to a meromorphic function on
$\C^m$, so that we may justify moving the contour of integration.

Let $\cdot$ denote the bilinear (not sesquilinear) dot product on
$\C^m$, and for $\lambda \in \C^m$ write $\lambda^2 := \lambda \cdot
\lambda$; this defines an analytic function from $\C^m$ to $\C$, and
$\lambda^2 = \abs{\lambda}^2$ iff $\lambda \in \R^m$.  For $w \in \C$,
let $\sqrt{w}$ denote the branch of the square root function
satisfying $\Im \sqrt{w} \ge 0$ and $\sqrt{w} > 0$ for $w > 0$ (so the
branch cut is the positive real axis).  Thus if $g : \C \to \C$ is an
analytic even function, $\lambda \mapsto g(\sqrt{\lambda^2})$ is
analytic as well, and satisfies $g(\sqrt{\lambda^2}) =
g(\abs{\lambda})$ for $\lambda \in \R^m$.  This holds in particular
for the function $\frac{\sinh w}{w}$, and thus the functions
$\frac{\sqrt{\lambda^2}}{\sinh \sqrt{\lambda^2}}$ and $\sqrt{\lambda^2} \coth
\sqrt{\lambda^2}$ are analytic away from points with $\sqrt{\lambda^2}
= ik\pi$, $k = 1, 2, \dots$.  

Using this notation, we let
\begin{align*}
  a_0(\lambda) &:= \left(\frac{\sqrt{\lambda^2}}{\sinh \sqrt{\lambda^2}}\right)^n \\
  a_1(\lambda) &:=  \cosh \sqrt{\lambda^2} \left(\frac{\sqrt{\lambda^2}}{\sinh \sqrt{\lambda^2}}\right)^{n+1} \\
  a_2(\lambda) &:= -i \left(\frac{\sqrt{\lambda^2}}{\sinh
    \sqrt{\lambda^2}}\right)^n \lambda \cdot \hat{z} \in \C^{2n}.
\end{align*}
As mentioned previously, $\hat{z}$ may be any unit vector in $\R^m$
without affecting the computation.  Therefore we shall treat it as
fixed, while $\abs{z}$ is allowed to vary.

Also, for $\lambda \in \C^m, \theta \in [0,\theta_0], \hat{z} \in
S^{m-1} \subset \R^m$, we define
\begin{equation}\label{fdef}
  f(\lambda, \theta, \hat{z}) := -{i \nu(\theta)} \lambda \cdot
  \hat{z} + \sqrt{\lambda^2} \coth \sqrt{\lambda^2}
\end{equation}
so that
\begin{equation*}
  \frac{\abs{x}^2}{4} f(\lambda, \theta(x,z), \frac{z}{\abs{z}}) =
  -i \lambda \cdot z + \frac{1}{4} \sqrt{\lambda^2} \coth
  \sqrt{\lambda^2} \abs{x}^2.
\end{equation*}
We henceforth write $\theta$ for $\theta(x,z)$.  Thus we now have
\begin{align}
  p_1(x,z) &= (4 \pi)^{-m-n} \int_{\R^m}
  e^{-\frac{\abs{x}^2}{4} f(\lambda, \theta, \hat{z})}
  a_0(\lambda)\,d\lambda \\
  q_i(x,z) &= (4 \pi)^{-m-n} \int_{\R^m}
  e^{-\frac{\abs{x}^2}{4} f(\lambda, \theta, \hat{z})}
  a_i(\lambda)\,d\lambda , && i=1,2
\end{align}
Written thus, the integrands have obvious meromorphic extensions to
$\lambda \in \C^n$, analytic away from the set $\{ \sqrt{\lambda^2}  =
ik\pi,\,k=1,2,\dots\}$.

A simple calculation verifies that $\frac{d}{dw} w \coth w = i
\nu(-iw)$, so we can compute the gradient of $f$ with respect to
$\lambda$ as
\begin{equation}
  \grad_\lambda f(\lambda, \theta, \hat{z}) = -i \nu(\theta) \hat{z} +
  i \nu(-i \sqrt{\lambda^2}) \hat{\lambda}
\end{equation}
which vanishes when $\lambda = i \theta \hat{z}$.  Thus
$i \theta \hat{z}$ is the desired critical point.  We observe that
\begin{equation}
  f(i \theta \hat{z}, \theta, \hat{z}) = \theta \nu(\theta) + i \theta
  \coth(i \theta) = \theta(\nu(\theta) + \cot(\theta)) = \frac{\theta^2}{\sin^2\theta}
\end{equation}
so by (\ref{distance-formula}),
\begin{equation}
  \abs{x}^2 f(i \theta \hat{z}, \theta, \hat{z}) = d(x,z)^2.
\end{equation}
Thus we define
\begin{equation}\label{psidef}
  \psi(\lambda, \theta, \hat{z}) := f(\lambda, \theta, \hat{z}) - f(i
  \theta \hat{z}, \theta, \hat{z}) = -i\nu(\theta)\lambda\cdot\hat{z}
  + \sqrt{\lambda^2}\coth\sqrt{\lambda^2} - \frac{\theta^2}{\sin^2\theta}.
\end{equation}
We then have
\begin{equation}
  p_1(x,z) = (4 \pi)^{-m-n}e^{-d(x,z)^2/4} \int_{\R^m}
  e^{-\frac{\abs{x}^2}{4} \psi(\lambda, \theta, \hat{z})} a_0(\lambda)\,d\lambda
\end{equation}
and analogous formulas for $q_1, q_2$.  Thus let
\begin{equation}\label{hi-def}
  h_i(x,z) := \int_{\R^m}
  e^{-\frac{\abs{x}^2}{4} \psi(\lambda, \theta, \hat{z})} a_i(\lambda)\,d\lambda.
\end{equation}
It will now suffice to estimate $h_i$.


The first step in the steepest descent method is to move the
``contour'' of integration to pass through $i \theta \hat{z}$.  Some
preliminary computations are in order.

\begin{lemma}\label{sab}
  For $a, b \in \R^m$, we have
  \begin{equation}
\abs{a} - \abs{b} \le \abs{\Re \sqrt{(a+bi)^2}} \le \abs{a}, \quad 0 \le \Im
\sqrt{(a+bi)^2} \le \abs{b}.
  \end{equation}
  Equality holds in the upper bounds if and only if $a$ and $b$ are
  parallel, i.e. $a=rb$ for some $r \in \R$.
\end{lemma}
\begin{proof}
  First note that $(a+bi)^2 = \abs{a}^2 - \abs{b}^2 + 2i a \cdot b$.
  So by the Cauchy-Schwarz inequality,
  \begin{equation}\label{ab2-upper}
    \begin{split}
    \abs{(a+bi)^2}^2 &= (\abs{a}^2 - \abs{b}^2)^2 + (2 a \cdot b)^2
    \\
    &\le (\abs{a}^2 - \abs{b}^2)^2 + 4 \abs{a}^2 \abs{b}^2 \\
    &= (\abs{a}^2+\abs{b}^2)^2
    \end{split}
\end{equation}
  so that $\abs{(a+bi)^2} \le \abs{a}^2+\abs{b}^2$.  Equality holds in
  the Cauchy-Schwartz inequality iff $a$ and $b$ are parallel.  On the
  other hand,
  \begin{equation}\label{ab2-lower}
  \abs{(a+bi)^2} \ge \Re (a+bi)^2 = \abs{a}^2 - \abs{b}^2.    
  \end{equation}
  Now we can write
  \begin{align*}
    \left(\Re \sqrt{(a+bi)^2}\right)^2 &= \frac{1}{4} \left(\sqrt{(a+bi)^2} +
    \conj{\sqrt{(a+bi)^2}}\right)^2 \\
    &= \frac{1}{4} \left((a+bi)^2 + \conj{(a+bi)^2} + 2
    \abs{\sqrt{(a+bi)^2}}^2\right) \\
    &= \frac{1}{2} (\abs{a}^2 - \abs{b}^2 + \abs{(a+bi)^2}).
  \end{align*}
  The upper bound for $\abs{\Re \sqrt{(a+bi)^2}}$ then follows
  from (\ref{ab2-upper}).  The lower bound is trivial if $\abs{a} \le
  \abs{b}$, and otherwise we have by (\ref{ab2-lower}) that
  \begin{equation*}
    \left(\Re \sqrt{(a+bi)^2}\right)^2 \ge \abs{a}^2 - \abs{b}^2 \ge
    (\abs{a} - \abs{b})^2.
  \end{equation*}
  The lower bound for $\Im \sqrt{(a+bi)^2}$ holds by our definition of
  $\sqrt{\cdot}$, and the upper bound is similar to the previous one.
\end{proof}

\begin{lemma}\label{integrand-bounds}
  For each $\theta_0 \in [0,\pi)$ there exists $c(\theta_0) > 0$ such that if $a,b \in \R^n$ with
  $\abs{a} \ge c(\theta_0)$, $\abs{b} \le 2\pi$, we have
  \begin{align}
    \Re \psi(a+ib, \theta, \hat{z}) &\ge \abs{a}/2 \\
    \intertext{and}
    \abs{a_i(a+ib)} &\le 1
  \end{align}
  for all $\theta \in [0,\theta_0]$, $\hat{z}, \hat{x} \in S^{m-1} \subset \R^{m}$.
\end{lemma}
\begin{proof}
  Fix $\theta_0 \in [0,\pi)$.  Note first that
  \begin{equation}
    \Re \psi(a + i b, \theta, \hat{z}) = \nu(\theta) b
    \cdot \hat{z} - \Re f(i \theta \hat{z}, \theta, \hat{z}) +
    \Re\left[\sqrt{(a+bi)^2} \coth \sqrt{(a+bi)^2}\right].
  \end{equation}
  By continuity, $\nu(\theta) b
    \cdot \hat{z} - \Re f(i \theta \hat{z}, \theta, \hat{z})$ is
    bounded below by some constant independent of $a$ for all $\theta
    \in [0,\theta_0]$, $\abs{b} \le 2\pi$.  Thus it suffices to
      show that for sufficiently large $\abs{a}$,
  \begin{equation}\label{foo}
    \Re\left[\sqrt{(a+bi)^2} \coth \sqrt{(a+bi)^2}\right] \ge \frac{2}{3}\abs{a}.
  \end{equation}
  Now for $\alpha \in \R$, $\beta \in [-2\pi,2\pi]$ we have
  \begin{align*}
  \Re((\alpha+i\beta) \coth (\alpha + i\beta)) &= \frac{\alpha \sinh \alpha \cosh \alpha + \beta
      \sin \beta \cos \beta}{\cosh^2 \alpha - \cos^2 \beta} \\
    &\ge \alpha \coth \alpha - \frac{\beta}{\cosh^2 \alpha} \\
    &\ge \alpha \coth \alpha - \frac{2\pi}{\cosh^2 \alpha} \\
  &\ge \frac{3}{4} \abs{\alpha}
  \end{align*}
  for sufficiently large $\abs{\alpha}$.  (Recall that $\lim_{\alpha
    \to \pm \infty} \coth\alpha = \pm 1$.)  Thus, since 
  \begin{align*}
    \abs{\Re
  \sqrt{(a+bi)^2}} &\ge \abs{a} - \abs{b} \ge \abs{a} - 2\pi
    \intertext{and}
    \abs{\Im \sqrt{(a+bi)^2}} &\le 2\pi,
  \end{align*}
 it is clear that (\ref{foo}) holds for sufficiently large $\abs{a}$.

  For the bound on $a_i$, note that the $\sinh$ factor in the
  denominator of each $a_i$ can be estimated by
  \begin{equation*}
    \abs{\sinh(\alpha+i\beta)} = \abs{\frac{e^{\alpha+i\beta} -
        e^{-{\alpha+i\beta}}}{2}} \ge \abs{\frac{\abs{e^{\alpha+i\beta}} -
        \abs{e^{-\alpha+i\beta}}}{2}} = \abs{\sinh \alpha}
  \end{equation*}
  so that $\abs{\sinh \sqrt{(a+bi)^2}} \ge \abs{\sinh \Re
    \sqrt{(a+bi)^2}} \ge \abs{\sinh(\abs{a}-2\pi)}$ for $\abs{a} \ge
  2\pi$.  This grows exponentially with $\abs{a}$, so it certainly dominates
  the polynomial growth of the numerator, and we have $\abs{a_i(a+ib)}
  \le 1$ for large enough $\abs{a}$.
\end{proof}

\begin{lemma}\label{move-contour-Rm}
  Let $F(\lambda) := e^{-\frac{\abs{x}^2}{4} \psi(\lambda, \theta, \hat{z})} a_i(\lambda)$ be the
  integrand in (\ref{hi-def}), where $x,z$ are fixed.  If $\tau \in
  \R^m$ with $\abs{\tau} < \pi$, then
  \begin{equation}
    h_i(x,z) = \int_{\R^m} F(\lambda)\,d\lambda = 
    \int_{\R^m} F(\lambda + i \tau)\,d\lambda.
  \end{equation}
\end{lemma}

\begin{proof}
  Note first that $F$ is analytic at $\lambda + ib$ when $\abs{b} <
  \pi$, by the second inequality in Lemma \ref{sab}.  Also, by Lemma
  \ref{integrand-bounds}, we have 
  \begin{equation}\label{F-bound}
\abs{F(\lambda+ib)} \le e^{-\abs{x}^2 \abs{\lambda}/8}
  \end{equation}
as soon as $\abs{\lambda} > c(\theta)$.

  We view $\int_{\R^m} F(\lambda)\,d\lambda$ as $m$ iterated integrals
  and handle them one at a time.  For $1 \le k \le m$, suppose we
  have shown that
  \begin{equation}
    \int_{\R^m} F(\lambda)\,d\lambda
    = \int_\R \dots \int_\R F(\lambda_1 + i\tau_1, \dots, \lambda_{k-1} +
    i\tau_{k-1}, \lambda_{k}, \dots, \lambda_m)\,d\lambda_1 \dots d\lambda_m.
  \end{equation}
  Continuity of $F$ and (\ref{F-bound}) show that $F$ is
  integrable, so we may apply Fubini's theorem and evaluate the
  $d\lambda_k$ integral first:
  \begin{equation*}
     \int_{\R^m} F(\lambda)\,d\lambda
    = \int_\R \dots \int_\R F(\lambda_1 +i \tau_1, \dots, \lambda_{k-1} +
   i \tau_{k-1}, \lambda_{k}, \dots, \lambda_m)\,d\lambda_k d\lambda_1 \dots d\lambda_m.
  \end{equation*}
  Now
  \begin{align*}
    &\quad \int_\R F(\lambda_1 + i\tau_1, \dots, \lambda_{k-1} +
    i\tau_{k-1}, \lambda_{k}, \dots, \lambda_m)\,d\lambda_k  \\
    &= \lim_{\alpha \to \infty} \int_{-\alpha}^{\alpha} F(\lambda_1 +i \tau_1, \dots, \lambda_{k-1} +
   i \tau_{k-1}, \lambda_{k}, \dots, \lambda_m)\,d\lambda_k.
  \end{align*}
  Since $\lambda_k \mapsto F(\lambda_1 + i\tau_1, \dots, \lambda_{k-1} +
    i\tau_{k-1}, \lambda_{k}, \dots, \lambda_m)$ is analytic for
    $\abs{\Im \lambda_k} \le \tau_k$ (which holds because $\abs{(\tau_1, \dots,
    \tau_k)} \le \abs{\tau} < \pi$), we have
    \begin{align*}
      \int_{-\alpha}^{\alpha} F(\dots, \lambda_{k}, \dots)\,d\lambda_k
      = \int_{-\alpha}^{-\alpha+i\tau_k} F +
      \int_{-\alpha+i\tau_k}^{\alpha+i\tau_k} F + \int_{\alpha+i\tau_k}^{\alpha} F
    \end{align*}
    where the contour integrals are taken along straight (horizontal
    or vertical) lines.  But as soon as $\alpha$ exceeds $c(\theta)$ from
    Lemma \ref{integrand-bounds}, (\ref{F-bound}) gives
\begin{align*}
&\quad \int_{-\alpha}^{-\alpha+i\tau_k} \abs{F(\lambda_1 + i\tau_1, \dots, \lambda_{k-1} +
    i\tau_{k-1}, \lambda_{k}, \dots, \lambda_m)}\,d\lambda_k \\ 
  &\le
  \tau_k e^{-\abs{x}^2 \abs{(\lambda_1, \dots, \lambda_{k-1},
      -\alpha, \lambda_k, \dots, \lambda_m)}/8} \\
  &\le \pi e^{-\abs{x}^2 \abs{\alpha}/8} \to 0 \text{ as } \alpha \to \infty.
\end{align*}
A similar argument shows the same for
$\int_{\alpha+i\tau_k}^{\alpha} F$, so we have
\begin{align*}
  &\quad\int_\R F(\lambda_1 + i\tau_1, \dots, \lambda_{k-1} +
  i\tau_{k-1}, \lambda_{k}, \dots, \lambda_m)\,d\lambda_k \\&=
  \int_{-\infty+i\tau_k}^{\infty+i\tau_k} F(\lambda_1 + i\tau_1, \dots, \lambda_{k-1} +
  i\tau_{k-1}, \lambda_{k}, \dots, \lambda_m)\,d\lambda_k \\
  &= \int_\R F(\lambda_1 + i\tau_1, \dots, \lambda_{k-1} +
  i\tau_{k-1}, \lambda_{k}+i\tau_k, \dots, \lambda_m)\,d\lambda_k.
\end{align*}
Thus applying Fubini's theorem again, we have shown
\begin{equation}
    \int_{\R^m} F(\lambda)\,d\lambda
    = \int_\R \dots \int_\R F(\lambda_1 + i\tau_1, \dots, \lambda_{k-1} +
    i\tau_{k-1}, \lambda_{k}+i\tau_k, \dots, \lambda_m)\,d\lambda_1 \dots d\lambda_m.
\end{equation}
Applying this argument successively for $k=1, 2, \dots, m$ establishes
the lemma.
\end{proof}

For the remainder of this section, we assume that $\abs{z} \le
B_1\abs{x}^2$, so that $\theta \le \theta_0(B_1)$. We next show that
the contribution from $\lambda$ far from the origin is negligible.

\begin{lemma}\label{far-from-origin}
  There exist $r > 0$ and a constant $C > 0$ such that
  \begin{equation}
    \abs{\int_{B(0,r)^C} e^{-\frac{\abs{x}^2}{4}\psi(\lambda + i \theta \hat{z}, x, z)} a_i(\lambda +
      i \theta \hat{z})\,d\lambda} \le \frac{C}{\abs{x}^{2m}}. 
  \end{equation}
\end{lemma}

\begin{proof}
  From Lemma \ref{integrand-bounds}, if $r \ge c(\theta_0)$ we have
  \begin{align*}
    \int_{B(0,r)^C} \abs{e^{-\frac{\abs{x}^2}{4}\psi(\lambda + i \theta \hat{z}, x, z)} a_i(\lambda +
      i \theta \hat{z})}\,d\lambda &\le \int_{B(0,r)^C}
    e^{-\frac{\abs{x}^2}{8}\abs{\lambda}}\,d\lambda \\
    &= \omega_{m-1} \int_r^\infty e^{-\abs{x}^2\rho/8} \rho^{m-1}\,d\rho \\
    &\le \omega_{m-1} \int_0^\infty e^{-\abs{x}^2\rho/8} \rho^{m-1}\,d\rho \\
    &= \omega_{m-1} (b \abs{x}^2)^{-m} \int_0^\infty e^{-\rho}
    \rho^{m-1}\,d\rho \\
    &= \frac{C}{\abs{x}^{2m}}
  \end{align*}
  where $\omega_{m-1}$ is the hypersurface measure of $S^{m-1}$.
\end{proof}

We can now apply a steepest descent argument.  As a similar argument
will be used later in this paper (see Proposition \ref{Fprop-III}), we
encapsulate it in the following lemma.

 \begin{lemma}\label{abstract-lemma-multi}
   Let $\Sigma \subset \R^k$ for some $k$, $r > 0$, $B(0,r)$ the ball of
   radius $r$ in $\R^m$, and $g : B(0,r) \times \Sigma \to \R$, $k :
   \R^{2n} \times [-r,r] \times \Sigma \to \C$ be measurable.  Define $F
   : \R^{2n} \times \Sigma \to \C$ by
   \begin{equation}
     F(x,\sigma) := \int_{B(0,r)} e^{-\abs{x}^2 g(\lambda, \sigma)} k(x,\lambda,\sigma)\,d\lambda.
   \end{equation}
   Suppose:
   \begin{compactenum}
   \item \label{abstract-b1} There exists a positive constant $b_1$ such that
     $g(\lambda,\sigma) \ge b_1 \abs{\lambda}^2$ for all $\lambda \in B(0,r), \sigma
     \in \Sigma$;  
   \item \label{abstract-k2} $k$ is bounded, i.e. $k_2 := \sup_{x \in \R^{2n}, \lambda \in
     B(0,r), \sigma \in \Sigma} \abs{k(x,\lambda,\sigma)} < \infty$.
     \setcounter{continuehere}{\value{enumi}}
   \end{compactenum}

   Then there exists a positive constant $C_2'$ such that
   \begin{equation}
     \abs{F(x,\sigma)} \le \frac{C_2'}{\abs{x}^m}
   \end{equation}
   for all $x > 0$, $\sigma \in \Sigma$.

   If additionally we have:
   \begin{compactenum}\setcounter{enumi}{\value{continuehere}}
   \item \label{abstract-b2} There exists a positive constant $b_2$ such that
     $g(\lambda,\sigma) \le b_2 \abs{\lambda}^2$ for all $\lambda \in B(0,r), \sigma
     \in \Sigma$;
   \item \label{abstract-k1} There exists a function $\epsilon : \R^+ \to [0,r]$ such that
     $\lim_{\rho \to +\infty} \rho \epsilon(\rho) = +\infty$, and
     \begin{equation}
       k_1 := \inf_{x \in \R^{2n}, \lambda \in B(0,\epsilon(\abs{x})),
         \sigma \in \Sigma} \Re k(x,\lambda,\sigma) > 0.
     \end{equation}
   \end{compactenum}

   Then there exist positive constants $C_1'$ and $x_0$ such that
   for all $\abs{x} \ge x_0$ and $\sigma \in \Sigma$ we have 
   \begin{equation}
     \Re F(x,\sigma) \ge \frac{C_1'}{\abs{x}^m}.
   \end{equation}
 \end{lemma}

 \begin{proof}
   The upper bound is easy, since
   \begin{align*}
     \abs{F(x,\sigma)} &\le k_2 \int_{B(0,r)} e^{-\abs{x}^2 b_1 \abs{\lambda}^2}\,d\lambda
     \\
     &= \frac{k_2}{\abs{x}^m} \int_{B(0,rx)} e^{-b_1 \abs{\lambda}^2}\,d\lambda \\
     &\le \frac{k_2}{\abs{x}^m} \int_{\R^m} e^{-b_1 \abs{\lambda}^2}\,d\lambda
     \\
     &= \frac{k_2 (\pi/b_1)^{m/2}}{\abs{x}^m}.
   \end{align*}
   For the lower bound, let
   \begin{align*}
     F_1(x,\sigma) &:= \int_{B(0,r) \backslash B(0,\epsilon(\abs{x}))}
     e^{-\abs{x}^2 g(\lambda, \sigma)} k(x,\lambda,\sigma)\,d\lambda \\
     F_2(x,\sigma) &:= \int_{B(0, \epsilon(\abs{x}))} e^{-\abs{x}^2 g(\lambda, \sigma)}
     k(x,\lambda,\sigma)\,d\lambda
   \end{align*}
   so that $F = F_1 + F_2$.  Now we have
   \begin{align*}
     \abs{F_1(x,\sigma)} &\le k_2 \int_{B(0,r) \backslash B(0,\epsilon(\abs{x}))} e^{-\abs{x}^2 b_1 \abs{\lambda}^2}\,d\lambda \\
     &\le k_2 \int_{\R^m \backslash B(0,\epsilon(\abs{x}))} e^{-\abs{x}^2 b_1
       \abs{\lambda}^2}\,d\lambda \\
     &\le \frac{k_2}{\abs{x}^m} \int_{\R^m \backslash B(0,\abs{x}\epsilon(\abs{x}))} e^{- b_1
       \abs{\lambda'}^2}\,d\lambda' \\
   \end{align*}
   where we make the change of variables $\lambda' = \abs{x} \lambda$.
   For $F_2$ we have
   \begin{align*}
     \Re{F_2(x,\sigma)} &\ge k_1 \int_{B(0,\epsilon(\abs{x}))}
       e^{-\abs{x}^2 b_2 \abs{\lambda}^2}\,d\lambda \\
       &= \frac{1}{\abs{x}^m} k_1 \int_{B(0, \abs{x}\epsilon(\abs{x}))} e^{-b_2 \abs{\lambda'}^2}\,d\lambda'.
   \end{align*}
 So we have
 \begin{align*}
   \abs{x}^m \Re F(x,\sigma) &\ge \abs{x}^m \Re F_2(x,\sigma) - \abs{\abs{x}^m F_1(x,\sigma)} \\
   &\ge k_1 \int_{B(0,\abs{x}\epsilon(\abs{x}))} e^{-b_2
   \abs{\lambda'}^2}\,d\lambda' - k_2 \int_{\R^m \backslash B(0,\abs{x}\epsilon(\abs{x}))} e^{- b_1 \abs{\lambda'}^2}\,d\lambda' \\
   &\to  k_1 (\pi/b_2)^{m/2} - 0 > 0
 \end{align*}
 as $\abs{x} \to \infty$.
 So there exists $x_0$ so large that for all $\abs{x} \ge x_0$,
 \begin{equation}\label{F-lower}
   \Re F(x,\sigma) \ge \frac{1}{2} k_1 (\pi/b_2)^{m/2} \frac{1}{\abs{x}^m}
 \end{equation}
 as desired.
 \end{proof}

We need another computation before being able to apply this lemma.

  \begin{lemma}\label{zcothz}
    $\Re \sqrt{(\lambda + i\theta\hat{z})^2} \coth \sqrt{(\lambda +
      i\theta\hat{z})^2} \ge \theta \cot \theta$, with equality iff
    $\lambda = 0$.
  \end{lemma}

\begin{proof}
  We first note that the function $\beta
  \cot \beta$ is strictly decreasing on $[0,\pi)$.  To see this, note
    $\frac{d}{d\beta} \beta \cot \beta = -\nu(\beta)$.  By Corollary
    \ref{nu-c} $\nu(\beta) > 0$.  In particular, $\beta \cot \beta \le 1$.

    Next we observe that for $\alpha \in \R$, $\beta \in [0,\pi)$ we have
  \begin{equation}\label{claim1}
    \Re((\alpha+i\beta) \coth (\alpha + i\beta)) \ge \beta \cot \beta
  \end{equation}
  with equality iff $\alpha = 0$.  This can be seen by verifying that
  \begin{equation}
    \Re((\alpha+i\beta) \coth (\alpha + i\beta)) - \beta \cot \beta =
    \frac{\sinh^2 \alpha (\alpha \coth \alpha - \beta \cot
      \beta)}{\cosh^2 \alpha - \cos^2 \beta}
  \end{equation}
  which is a product of positive terms when $\alpha \ne 0$, since
  $\alpha \coth \alpha > 1 \ge \beta \cot \beta$ and $\cosh^2 \alpha > 1
  \ge \cos^2\beta$.
 
  Therefore, we have
    \begin{align}
      \Re \sqrt{(\lambda + i\theta\hat{z}^2} \coth \sqrt{(\lambda +
      i\theta\hat{z})^2}  &\ge \left(\Im \sqrt{(\lambda +
        i\theta\hat{z})^2}\right) \cot \left(\Im \sqrt{(\lambda +
        i\theta\hat{z})^2}\right) \label{claim-ineq1}\\
      &\ge \theta \cot \theta \label{claim-ineq2}
    \end{align}
    because $0 \le \Im \sqrt{(\lambda +
        i\theta\hat{z})^2} \le \theta < \pi$ by Lemma \ref{sab}.

    If equality holds in (\ref{claim-ineq2}), it must be that $\Im
    \sqrt{(\lambda + i\theta\hat{z})^2} = \theta$.  By Lemma \ref{sab}
    $\lambda$ and $\hat{z}$ are parallel, so $\sqrt{(\lambda +
      i\theta\hat{z})^2} = \pm\abs{\lambda} + i \theta$.  If equality also
    holds in (\ref{claim-ineq1}), we have
    \begin{equation*}
            \Re (\pm\abs{\lambda} + i\theta) \coth(\pm\abs{\lambda} +
            i\theta) = \theta \cot \theta
    \end{equation*}
    so by (\ref{claim1}) it must be that $\abs{\lambda}=0$.  This
    proves the claim.
\end{proof}

\begin{lemma}\label{psi-quadratic}
  Given $r > 0$, there exist constants $b_1, b_2, b_3 > 0$ depending only on
  $r$ and $\theta_0$ such that
  \begin{align}
    b_1 \abs{\lambda}^2 \le \Re \psi(\lambda + i \theta \hat{z}, \theta,
    \hat{z}) &\le b_2 \abs{\lambda}^2 \label{re-psi-quadratic} \\
\intertext{and}
    \abs{\Im \psi(\lambda + i \theta \hat{z}, \theta,
    \hat{z})} &\le b_3 \abs{\lambda}^3 \label{im-psi-cubic}
  \end{align}
  for all $\lambda \in B(0,r) \subset \R^m, \theta \in [0,\theta_0], \hat{z} \in
  S^{m-1} \subset \R^m$.
\end{lemma}

\begin{proof}
  Note first that $\psi(\lambda+i\theta\hat{z}, \theta, \hat{z})$ is
  smooth for $\theta \in [0,\theta_0]$ since $\Im
  \sqrt{(\lambda+i\theta\hat{z})} \le \theta \le \theta_0 < \pi$, so
  that we are avoiding the singularities of $w \coth w$.

  We have $\psi(i\theta\hat{z}, \theta, \hat{z}) = 0$ and
  $\grad_{\lambda} \psi(i\theta\hat{z}, \theta, \hat{z}) = 0$.  We now
  show the Hessian $H(i \theta \hat{z})$ of $\psi$ at $i\theta\hat{z}$
  is real and uniformly positive definite.

  By direct computation, we can find
  \begin{equation}
    \frac{\partial^2}{\partial \lambda_i \partial \lambda_j}
    \psi(\lambda, \theta, \hat{z}) = \nu'(-i \sqrt{\lambda^2})
    \frac{\lambda_i \lambda_j}{\lambda^2} + i \frac{\nu(-i\sqrt{\lambda^2})}{\sqrt{\lambda^2}}
    \left({\delta_{ij}} - \frac{\lambda_i
      \lambda_j}{\lambda^2}\right)
  \end{equation}
  so that for $u \in \R^m$,
  \begin{equation}
    H(\lambda) u \cdot u = \nu'(-i \sqrt{\lambda^2})
    \frac{(\lambda \cdot u)^2}{\lambda^2} + i \frac{\nu(-i\sqrt{\lambda^2})}{\sqrt{\lambda^2}}
    \left(\abs{u}^2 - \frac{(\lambda \cdot u)^2}{\lambda^2}\right)
  \end{equation}
  and in particular
  \begin{align*}
    H(i \theta \hat{z}) u \cdot u &= \nu'(\theta)
    (\hat{z} \cdot u)^2 + \frac{\nu(\theta)}{\theta}
    \left(\abs{u}^2 - \hat{z} \cdot u)^2\right) \\
    &= \abs{u}^2 \left(s \nu'(\theta)
    + \frac{\nu(\theta)}{\theta}(1-s)\right)
  \end{align*}
  where $s := \left(\frac{\hat{z} \cdot u}{\abs{u}}\right)^2$, so $0
  \le s \le 1$.  Note this is a real number whenever $u \in \R^m$.
  Thus we have $H(i\theta\hat{z}) u \cdot u$ written as a convex
  combination of two real functions of $\theta$, so
  \begin{equation}
    H(i\theta\hat{z}) u \cdot u \ge \abs{u}^2 
    \min\{\frac{\nu(\theta)}{\theta}, \nu'(\theta)\} \ge c \abs{u}^2 
  \end{equation}
  where $c$ is the lesser of the two constants provided by Lemma
  \ref{nu-increase} and Corollary \ref{nu-c} respectively.  This is
  valid for $\theta > 0$ and hence by continuity also for $\theta =
  0$.

  By Taylor's theorem, this shows that (\ref{re-psi-quadratic}) and
  (\ref{im-psi-cubic}) hold for small $\lambda$.  The upper bounds
  thus automatically hold for all $\lambda \in B(0,r)$ by continuity.
  To obtain the lower bound on $\Re \psi$, it will suffice to show
  $\Re \psi > 0$ for all $\lambda \ne 0$.  But we have 
    \begin{align*}
      \Re \psi(\lambda + i \theta \hat{z}, \theta, \hat{z}) &= \theta \nu(\theta) - \Re f(i \theta \hat{z}, \theta, \hat{z}) +
      \Re\left[\sqrt{(\lambda+i\theta\hat{z})^2} \coth
        \sqrt{((\lambda+i\theta\hat{z})^2}\right] \\
      &= \theta \nu(\theta) - \frac{\theta^2}{\sin^2\theta} +
      \Re\left[\sqrt{(\lambda+i\theta\hat{z})^2} \coth
        \sqrt{((\lambda+i\theta\hat{z})^2}\right] \\
      &= -\theta \cot \theta + \Re\left[\sqrt{(\lambda+i\theta\hat{z})^2} \coth
        \sqrt{((\lambda+i\theta\hat{z})^2}\right] \\ &\ge 0
    \end{align*}
    by Lemma \ref{zcothz}, with equality iff $\lambda = 0$.
\end{proof}

The proof of Theorem \ref{region-I-theorem} can now be completed.

\begin{proof}[Proof of Theorem \ref{region-I-theorem}]
  
We establish (\ref{I-p-both}) first.  We can apply Lemma
\ref{abstract-lemma-multi} with $\Sigma := [0, \theta_0] \times
S^{m-1}$, $\sigma = (\theta, \hat{z})$, $r$ the value from Lemma
\ref{far-from-origin}, and
\begin{align*}
  g(\lambda, (\theta, \hat{z}))
&:= \frac{1}{4} \Re\psi(\lambda + i \theta \hat{z}, \theta, \hat{z}) \\
  k(x, \lambda, (\theta, \hat{z})) &:= e^{i \frac{\abs{x}^2}{4}
    \Im\psi(\lambda + i \theta \hat{z}, \theta, \hat{z})} a_0(\lambda+i\theta\hat{z}).
\end{align*}
The necessary bounds on $g$ come from (\ref{re-psi-quadratic}).  For
an upper bound on $k$, we have $\abs{k(x, \lambda, (\theta, \hat{z}))}
= \abs{a_0(\lambda+i\theta\hat{z})}$, which is bounded by the fact
that $(\lambda, \theta, \hat{z})$ ranges over the bounded region
$B(0,r) \times [0,\theta_0] \times S^{m-1}$ which avoids the
singularities of $a_0$.

Now for the lower bound on $k$.  By direct
computation, we have $a_0(i \theta \hat{z}) =
\left(\frac{\theta}{\sin\theta}\right)^n \ge 1$; by continuity there
exists $\delta$ such that $\Re e^{is} a_0(\lambda +i \theta \hat{z}) \ge \frac{1}{2}$
for all $\abs{\lambda} \le \delta$ and $\abs{s} \le \delta$, where $s
\in \R$.  If $\abs{\lambda} \le \abs{x}^{-2/3} \delta / b_3$, where
$b_3$ is as in (\ref{im-psi-cubic}), we will have $\abs{x}^2 \abs{\Im
  \psi(\lambda + i \theta \hat{z})} \le \delta$.  Thus set
$\epsilon(x) := \min\{\delta, \abs{x}^{-2/3} \delta / b_3\}$, so that
$\Re k(x, \lambda, (\theta, \hat{z}) \ge \frac{1}{2}$ for all $\abs{\lambda}
\le \epsilon(x)$ and all $(\theta, \hat{z}) \in \Sigma$, and
$\lim_{\rho \to \infty} \rho \epsilon(\rho) = \lim_{\rho \to \infty}
\rho^{1/3} \delta / b_3 = +\infty$.

Thus Lemma \ref{abstract-lemma-multi} applies, and so combining it
with Lemmas \ref{move-contour-Rm} and \ref{far-from-origin} we have
that there exist positive constants $C, C_1', C_2', x_0$ such that
\begin{equation}
  \left(\frac{C_1'}{\abs{x}^m} - \frac{C}{\abs{x}^{2m}}\right)
  e^{-\frac{1}{4}d(x,z)^2} \le p_1(x,z) \le \left(\frac{C_2'}{\abs{x}^m} + \frac{C}{\abs{x}^{2m}}\right)
  e^{-\frac{1}{4}d(x,z)^2}.
\end{equation}
whenever $\abs{x} \ge x_0$.  We can choose $x_0$ larger if necessary
so that $\abs{x}^{-m} \gg \abs{x}^{-2m}$.  Then taking $d_0 = x_0$
will establish (\ref{I-p-both}).

For $q_i$, the upper bound is similar; $\abs{a_i}$ is bounded above
just like $\abs{a_0}$, establishing (\ref{I-q-upper}).

For (\ref{I-q-lower}), we cannot necessarily bound both $\abs{q_i}$
below simultaneously, but it suffices to take them one at a time.  For
$0 \le \theta(x,z) \le \frac{\pi}{4}$, we have $a_1(i \theta \hat{z})
= \cos \theta \left(\frac{\theta}{\sin\theta}\right)^{n+1} \ge
\frac{1}{\sqrt{2}}$, so by the above logic we obtain the desired lower
bound on $\abs{q_1}$ for such $\theta$.  If $\frac{\pi}{4} \le \theta
\le \theta_0$, we estimate $q_2$ in the same way, since we have $a_2(i
\theta \hat{z}) = \left(\frac{\theta}{\sin\theta}\right)^n \theta \ge
\frac{\pi}{4}.$
\end{proof}

\section{Polar coordinates}\label{polar-sec}

In this section, we obtain estimates for $p_1(x,z)$ and $\abs{\grad
  p_1(x,z)}$ when $\abs{z} \ge B_1 \abs{x}^2$, where $B_1$ is
sufficiently large.  This means that $\theta(x,z) \ge \theta_0$ for
some $\theta_0$ near $\pi$.  Note that by Corollary
\ref{distance-estimate}, we have  $d(x,z) \asymp \sqrt{\abs{z}}$ in
this region.

We first consider $p_1$ and show the following.

\begin{theorem}\label{region-II-III-theorem}
  For $m$ odd, there exist constants $B_1, d_0$ such that
  \begin{equation}
    p_1(x,z) \asymp
    \frac{\abs{z}^{n-\frac{m+1}{2}}}{1+(\abs{x}\sqrt{\abs{z}})^{n-\frac{1}{2}}}
    e^{-\frac{1}{4}d(x,z)^2}
  \end{equation}
  or, equivalently,
  \begin{equation}
  p_1(x,z) \asymp \frac{d(x,z)^{2n-m-1}}{1+(\abs{x}d(x,z))^{n-\frac{1}{2}}}
    e^{-\frac{1}{4}d(x,z)^2}
  \end{equation}
  for $\abs{z} \ge B_1 \abs{x}^2$ and $\abs{z} \ge d_0$
  (equivalently, $d(x,z) \ge d_0$).
\end{theorem}

The effect of the requirement that $\abs{z} \le B_1 \abs{x}^2$ in the
previous section was to ensure that the critical point
$i\theta\hat{z}$ stayed away from the singularities of the integrand.
As $B_1 \to \infty$, the critical point approaches the set of
singularities, and the change of contour we used is no longer
effective; the constants in the estimates of Theorem
\ref{region-I-theorem} blow up.  In the case of the Heisenberg groups,
where the center of $G$ has dimension $m=1$, the singularity is a
single point, and the technique used in \cite{hueber-muller} and
\cite{bgg} is to move the contour past the singularity and concentrate
on the resulting residue term.  For $m > 1$, the singularities form a
large manifold and this technique is not easy to use directly.
However, by making a change to polar coordinates, we can reduce the
integral over $\R^m$ to one over $\R$; this replaces the Fourier
transform by the so-called Hankel transform.  (A similar approach is
used in \cite{randall} in the context of $L^p$ estimates for the
analytic continuation of $p_t$.)  When $m$ is odd, we recover a
formula very similar to that for $m=1$, and the above-mentioned
technique is again applicable.

For the rest of this section, we assume that $m$ is odd.

For $m \ge 3$, we write (\ref{Rm-integral}) in polar coordinates to obtain
\begin{align}
  p_1(x,z) &= (2\pi)^{-m}(4\pi)^{-n} \int_0^\infty \int_{S^{m-1}} e^{{i}
    \rho \sigma \cdot z}\,d\sigma e^{-\frac{\abs{x}^2}{4} \rho \coth
    \rho} \left(\frac{\rho}{\sinh\rho}\right)^n \rho^{m-1} \,d\rho \\
  &= \frac{(2\pi)^{-m}(4\pi)^{-n}}{2} \int_{-\infty}^\infty \int_{S^{m-1}} e^{{i}
    \rho \sigma \cdot z}\,d\sigma e^{-\frac{\abs{x}^2}{4} \rho \coth
    \rho} \left(\frac{\rho}{\sinh\rho}\right)^n \rho^{m-1} \,d\rho
\end{align}
since the integrand is an even function of $\rho$.  (To see this, make
the change of variables $\sigma \to -\sigma$ in the $d\sigma$
integral.  It is not true when $m$ is even.)

The $d\sigma$ integral can be written in terms of a Bessel function.
Using spherical coordinates, we can write, for arbitrary $\hat{v} \in
S^{m-1}$ and $w \in \C$,
\begin{align*}
  \int_{S^{m-1}} e^{i w \sigma \cdot \hat{v}}\,d\sigma
  &= \frac{2 \pi^{\frac{m-1}{2}}}{\Gamma\left(\frac{m-1}{2}\right)}
  \int_{0}^{\pi} e^{iw \cos\varphi} \sin^{m-2}\varphi \,d\varphi \\
  &= \frac{4 \pi^{\frac{m-1}{2}}}{\Gamma\left(\frac{m-1}{2}\right)}
  \int_{0}^{\frac{\pi}{2}} \cos(w \cos\varphi) \sin^{m-2}\varphi
  \,d\varphi && \text{(by symmetry)}\\
  &= \frac{(2\pi)^{m/2}}{w^{m/2-1}}J_{m/2-1}(w) 
\intertext{(see page 79 of \cite{magnus})}
  &= \Re \frac{(2\pi)^{m/2}}{w^{m/2-1}} H_{m/2-1}^{(1)}(w)
\end{align*}
where $H_\nu(w)$ is the Hankel function of the first kind, defined by
$H_\nu(w) = J_\nu(w) + i Y_\nu(w)$, with $Y_\nu$ the Bessel function
of the second kind.  Page 72 of \cite{magnus}
has a closed-form expression for $H_\nu$ which yields
\begin{equation}\label{hankel-expansion}
  S_m(w) =  2(2\pi)^{\frac{m-1}{2}} \Re \left[\frac{e^{iw}}{w^{m-1}}
  \sum_{k=1}^{\frac{m-1}{2}} c_{m,k} (-iw)^k\right]
\end{equation}
where the coefficients are
\begin{align*}
  c_{m,k} &=  \frac{(m-k-2)!}{ 2^{\frac{m-1}{2}-k}
    \left(\frac{m-1}{2}-k\right)! (k-1)!} > 0.
\end{align*}

The reason for the use of the Hankel function is the appearance of the
$e^{iw}$ factor, which gives us an integrand looking much like that
for $p_t$ when $m=1$.  This will allow us to apply similar techniques
to those which have been used previously for $m=1$.
We have
\begin{align}
  p_1(x,z) &= (\Re) \sum_{k=1}^{(m-1)/2} c_{m,k} \abs{z}^{k-m+1} \int_{-\infty}^\infty
  e^{{i} \rho \abs{z} -\frac{\abs{x}^2}{4} \rho \coth
    \rho} \frac{\rho^{n}}{\sinh^n\rho} (-i\rho)^{k}
  \,d\rho \label{p-expansion} \\
  &= \sum_{k=1}^{(m-1)/2} c_{m,k} \abs{z}^{k-m+1} e^{-\frac{1}{4}d(x,z)^2}\int_{-\infty}^\infty
  e^{-\frac{\abs{x}^2}{4} \psi(\rho, \theta)} a_k(\rho)
  \,d\rho \label{p-expansion-phi}
\end{align}
where, using similar notation as before,
\begin{align}
  \psi(\rho, \theta) &:= -i \nu(\theta) \rho + \rho \coth \rho -
  \frac{\theta^2}{\sin^2\theta} \\
  a_k(\rho) &:= \left(\frac{\rho}{\sinh\rho}\right)^n (-i\rho)^k
\end{align}
The constants and coefficients have all been absorbed into the
$c_{m,k}$; we note that $c_{1,0} > 0$, $c_{m,k} > 0$ for $k \ge 1$,
and $c_{m,0}=0$ for $m > 1$. We dropped the $(\Re)$ because the
imaginary part vanishes, being the integral of an odd function.

For $m=1$, we can write
\begin{equation}\label{p-expansion-m1}
  p_1(x,z) = (4\pi)^{-n} e^{-\frac{1}{4}d(x,z)^2}\int_{-\infty}^\infty
  e^{-\frac{\abs{x}^2}{4} \psi(\rho, \theta)} a_0(\rho)
  \,d\rho
\end{equation}

The integrals appearing in the terms of the sum in
(\ref{p-expansion-phi}), as well as in (\ref{p-expansion-m1}), are all
susceptible to the same estimate, as the following theorem shows.

\begin{theorem}\label{h-estimate}
  Let $D \subset \C$ be the strip $D = \{ 0 \le \Im \rho \le 3
  \pi/2\}$.  Suppose $a(\rho)$ is a function analytic on $D \backslash
  \{i\pi\}$, with a pole of order $n$ at $\rho = i\pi$, $a(i \theta)
  \ge 1$ for $\theta_0 \le \theta < \pi$, and $\int_\R
  \abs{a(\rho+3i\pi/2)}\,d\rho < \infty$.  Let
  \begin{equation}
    h(x,z) := \int_{-\infty}^\infty e^{-\frac{\abs{x}^2}{4} \psi(\rho,
      \theta)} a(\rho) \,d\rho.
  \end{equation}
 There exist $B_1, d_0$ such that
 \begin{equation}
\Re  h(x,z) \asymp \frac{\abs{z}^{n-1}}{1+(\abs{x}\sqrt{\abs{z}})^{n-\frac{1}{2}}}
 \end{equation}
 for all $(x,z)$ with $\abs{z} \ge B_1 \abs{x}^2$ and $\abs{z} \ge d_0$.
\end{theorem}

The proof of Theorem \ref{h-estimate} occupies the rest of this
section.  Theorem \ref{region-II-III-theorem} follows, since Theorem
\ref{h-estimate} applies to each term of (\ref{p-expansion-phi}) (note
each $a_k$ satisfies the hypotheses), and the $k=(m-1)/2$ term will
dominate for large $\abs{z}$.


An argument similar to Lemma \ref{move-contour-Rm}, using the fact
that Lemma \ref{integrand-bounds} applies for $\abs{b} \le 2\pi$, will
allow us to move the contour to the line $\Im \rho = 3\pi/2$,
accounting for the residue at $i\pi$:
\begin{equation}
      h(x,z) := \underbrace{\int_{-\infty}^\infty e^{-\frac{\abs{x}^2}{4} \psi(\rho+3i\pi/2,
      \theta)} a(\rho+3i\pi/2) \,d\rho}_{h_l(x,z)} + \underbrace{\Res(e^{-\frac{\abs{x}^2}{4} \psi(\rho,
      \theta)} a(\rho); \rho = i\pi)}_{h_r(x,z)}.
\end{equation}

The following lemma shows that $h_l(x,z)$, the integral along the
horizontal line, is negligible.

 \begin{lemma}\label{hl-estimate}
   There exists $\theta_0 < \pi$ and a constant $C > 0$ such that for
   all $(x,z)$ with $\theta(x,z) \in [\theta_0,\pi)$ we have
   \begin{equation}
     \abs{h_l(x,z)} \le C e^{-d(x,z)^2/8}.
   \end{equation}
 \end{lemma}

 \begin{proof}
   Observe that $\coth(\rho+3 i
 \pi/2)=\tanh \rho$.  So
 \begin{align*}
   \Re\psi(\rho+3i\pi/2,\theta) &= \rho \tanh \rho + \frac{3\pi}{2}
   \nu(\theta) - \frac{\theta^2}{\sin^2\theta}
 \end{align*}
 Therefore we have
 \begin{align*}
   \abs{h_l(x,z)} &\le e^{-\frac{\abs{x}^2}{4} \left(\frac{3\pi}{2}\nu(\theta)
   - \frac{\theta^2}{\sin^2\theta}\right)} \int_\R
   e^{-\frac{\abs{x}^2}{4} \rho \tanh \rho}
   \abs{a(\rho+3i\pi/2)} \,d\rho \\
   &\le e^{-\frac{\abs{x}^2}{4} \left(\frac{3\pi}{2}\nu(\theta)
   - \frac{\theta^2}{\sin^2\theta}\right)} \int_\R
   \abs{a(\rho+3i\pi/2)} \,d\rho
 \end{align*}
 as $\tau\tanh\tau \ge 0$.  The integral in the last line is a finite
 constant, since $a(\cdot + 3i\pi/2)$ is integrable by assumption.

 However, for $\theta$ sufficiently close to $\pi$, we have
 $\nu(\theta) \ge \frac{1}{\pi} \frac{\theta^2}{\sin^2\theta}$.  (If
 $\beta(\theta) := \nu(\theta)
 \left(\frac{\theta^2}{\sin^2\theta}\right)^{-1}$, we have
 $\lim_{\theta \uparrow \pi} \beta(\theta) = 1/\pi$ and $\lim_{\theta
   \uparrow \pi} \beta'(\theta) = -2/\pi^2 < 0$.  Indeed, $\theta >
 0.51$ suffices.)  Thus for such $\theta$ we have
 \begin{equation}
   \abs{h_l(x,z)} \le C e^{-\frac{\abs{x}^2}{8}
     \frac{\theta^2}{\sin^2\theta}} = C e^{-d(x,z)^2/8}. 
 \end{equation}
 \end{proof}

\newcommand{\radius}{r}
\newcommand{\dummy}{\rho}
\newcommand{\oldlambda}{y}
\newcommand{\circlevar}{w}
\newcommand{\indexvar}{k}

 To handle the residue term $h_r$, write it as
 \begin{equation}\label{loop-integral}
   h_r(x,z) = \oint_{\partial B(i\pi, \radius)}
   e^{-\frac{\abs{x}^2}{4}\psi(\dummy,\theta)/4}(\dummy)\,d\dummy.
 \end{equation}
 We can choose any $\radius \in (0,\pi)$ because the integrand is analytic
 on the punctured disk.  To facilitate dealing with the singularity at
 $\theta = \pi$, we adopt the parameters
 \begin{equation}\label{lambda-s-def}
   \begin{split}
     s &:= \pi - \theta(x,z) \\
     \oldlambda &:= \pi \abs{x}^2/s.
   \end{split}
 \end{equation}
 Note that 
\begin{equation}\label{z-comparisons}
  \oldlambda/s \asymp \abs{z},\quad \oldlambda \asymp
  \abs{x}\sqrt{\abs{z}}.
\end{equation}
If we let (compare (\ref{psidef}))
\begin{align}
  \phi(\circlevar, s) &:= \frac{1}{4\pi}s
    \psi(i(\pi-\circlevar),\pi-s) \nonumber \\
    &= \frac{s}{4\pi}
    \left(\nu(\pi-s)(\pi-\circlevar) + (\pi-\circlevar)\cot(\pi-\circlevar)
    - \frac{(\pi-s)^2}{\sin^2 s}\right)\label{phidef} \\
  F(\oldlambda,s) &:= s^{n-1} \oint_{\partial B(0,\radius)}
  e^{-\oldlambda \phi(\circlevar, s)}
  a(i(\pi-\circlevar))(-i)\,d\circlevar \label{Fdef}
\end{align}
we have
\begin{equation}\label{hr-F}
  h_r(x,z) = s^{-(n-1)} F(y,s).
\end{equation}
Note we have made the change of variables $\dummy = i(\pi-\circlevar)$ from
(\ref{loop-integral}) to (\ref{Fdef}).  

Observe that $F$ is analytic in $\oldlambda$ and $s$ for $s \ne k\pi$,
$k \in \Z$, so we shall now consider $\oldlambda$ and $s$ as complex
variables.  The factor of $s^{n-1}$ in $F$ was inserted to clear a
pole of order $n-1$ at $s=0$, whose presence will be apparent later.

 Computing a Laurent series for $\phi$
 about $(i\pi,\pi)$, which converges for $0 < \abs{s} < \pi$, $0 <
 \abs{\circlevar} < \pi$, we find
 \begin{equation}\label{phi-series}
   \phi(\circlevar, s) = \frac{1}{2} - \frac{\circlevar}{4s} -
   \frac{s}{4\circlevar} - sU(\circlevar,s)
 \end{equation}
 with $U$ analytic for $\abs{s} < \pi$, $\abs{\circlevar} < \pi$.
 Also, by the hypotheses on $a$,
 \begin{equation}\label{a-series}
   a(i(\pi-\circlevar)) = \circlevar^{-n} V(\circlevar)
 \end{equation}
 where $V$ is analytic for $\abs{\circlevar} < \pi/2$ and
 $V(0)> 0$.  Thus we have
 \begin{equation}\label{FUV}
   F(\oldlambda,s) = s^{n-1} \oint_{\partial B(0,\radius)} e^{-\oldlambda\left(\frac{1}{2}- \frac{\circlevar}{4s} -
   \frac{s}{4\circlevar} - sU(\circlevar,s)\right)} \circlevar^{-n} V(\circlevar) (-i)\,d\circlevar
 \end{equation}

 The constant term in the expansion of $\psi$ is slightly inconvenient, so let
 $G(\oldlambda,s) = e^{\oldlambda/2} F(\oldlambda,s)$.  Then:
 \begin{align}
   G(\oldlambda,s) &= s^{n-1} \oint_{\partial B(0,\radius)}
     e^{\oldlambda\left(\frac{\circlevar}{4s} + \frac{s}{4\circlevar} + sU(\circlevar,s)\right)}
     \circlevar^{-n} V(\circlevar) (-i) \,d\circlevar \nonumber \\
   &= s^{n-1} \oint \sum_{\indexvar=0}^{\infty}
   \frac{\oldlambda^\indexvar}{\indexvar!} \left(\frac{\circlevar}{4s} +
   \frac{s}{4\circlevar} + sU(\circlevar,s)\right)^\indexvar \circlevar^{-n} V(\circlevar)\,d\circlevar (-i) \label{fubini-used} \\
   &= s^{n-1} \sum_{\indexvar=0}^{\infty}
   \frac{\oldlambda^\indexvar}{\indexvar!} \oint  \left(\frac{\circlevar}{4s} +
   \frac{s}{4\circlevar} + sU(\circlevar,s)\right)^\indexvar \circlevar^{-n} V(\circlevar) (-i) \,d\circlevar \nonumber\\ 
   &=: \sum_{\indexvar=0} \frac{\oldlambda^\indexvar g_\indexvar(s)}{\indexvar!} \label{Gsum}
 \end{align}
 where we let 
 \begin{equation}\label{g_m-def}
   g_\indexvar(s) := s^{n-1} \oint  \left(\frac{\circlevar}{4s} +
   \frac{s}{4\circlevar} + sU(\circlevar,s)\right)^\indexvar \circlevar^{-n} V(\circlevar) (-i) \,d\circlevar.
 \end{equation}
 The interchange of sum and integral in (\ref{fubini-used}) is
 justified by Fubini's theorem, since for fixed $s$ $U(s,\cdot)$ and
 $V$ are bounded on $B(0,\radius)$, and thus
 \begin{align*}
   &\quad \sum_{\indexvar=0}^{\infty} \oint_{B(0, \radius)}
   \abs{\frac{\oldlambda^\indexvar}{\indexvar!} \left(\frac{\circlevar}{4s} +
     \frac{s}{4\circlevar} + sU(\circlevar,s)\right)^\indexvar \left(\frac{\pi}{\circlevar} +
     V(\circlevar)\right)^n}\,d\circlevar \\
    &\le \sum_{\indexvar=0}^{\infty} \frac{\abs{\oldlambda}^\indexvar}{\indexvar!} 2 \pi \radius \left(\frac{\radius}{4\abs{s}} +
    \frac{\abs{s}}{4\radius} + \abs{s} \sup_{\abs{\circlevar}=\radius} \abs{U(\circlevar,s)}\right)^\indexvar \left(\frac{\pi}{\radius} +
    \sup_{\abs{\circlevar}=\radius} \abs{V(\circlevar)} \right)^n \\
    &= 2 \pi \radius  \left(\frac{\pi}{\radius} +
    \sup_{\abs{\circlevar}=\radius} \abs{V(\circlevar)}\right)^n \exp\left(\abs{\oldlambda} \left(\frac{\radius}{4\abs{s}} +
    \frac{\abs{s}}{4\radius} + \abs{s}\sup_{\abs{\circlevar}=\radius} \abs{U(\circlevar,s)}\right)\right) < \infty.
 \end{align*}

 We now examine more carefully the terms $g_\indexvar$ in (\ref{Gsum}--\ref{g_m-def}).
 \begin{lemma}\label{g_m}
   If $g_\indexvar$ is defined by (\ref{g_m-def}), then:
   \begin{enumerate}
     \item \label{g_m-analytic} $g_\indexvar$ is analytic for $\abs{s} \le s_0$;
     \item \label{g_m-size} There exists $C = C(s_0) \ge 0$ independent of $\indexvar$
     such that $\abs{g_\indexvar(s)} \le C^\indexvar$ for each $\indexvar$ and all $\abs{s} \le
     s_0$;
     \item \label{g_m-zero-low} For $\indexvar \le n-1$, $g_\indexvar(s) = s^{n-1-\indexvar}
     h_\indexvar(s)$, where $h_\indexvar$ is analytic for $\abs{s} \le s_0$.  In
     particular, $g_\indexvar(0)=0$ for $\indexvar < n-1$.
     \item \label{g_m-zero-high} For $\indexvar \ge n-1$, $g_\indexvar(0) > 0$ when
       $\indexvar+n$ is odd, and $g_\indexvar(0)=0$ when $\indexvar+n$ is even.
   \end{enumerate}
 \end{lemma}

 \begin{proof}
   By the multinomial theorem,
   \begin{align}
     g_\indexvar(s) &= \sum_{a+b+c=\indexvar} \binom{\indexvar}{a,b,c} s^{n-1} \oint_{\partial B(0,\radius)}
     \left(\frac{\circlevar}{4s}\right)^a 
     \left(\frac{s}{4\circlevar}\right)^b 
     \left(sU(\circlevar,s)\right)^c \circlevar^{-n} V(\circlevar) (-i)\,d\circlevar
     \\
     &= \sum_{a+b+c=\indexvar} \binom{\indexvar}{a,b,c} 4^{-(a+b)} \oint_{\partial B(0,\radius)}
     \circlevar^{a-b-n} s^{-(a-b-n)-1} \left(sU(\circlevar,s)\right)^c V(\circlevar) (-i) \,d\circlevar
     \\
     &= \sum_{\substack{a+b+c=\indexvar\\a-b-n \le -1}} \binom{\indexvar}{a,b,c} 4^{-(a+b)} \oint_{\partial B(0,\radius)}
     \circlevar^{a-b-n} s^{-(a-b-n)-1} \left(sU(\circlevar,s)\right)^c V(\circlevar) (-i) \,d\circlevar \label{abc}
   \end{align}
   since for terms with $a-b-n \ge 0$, the integrand is analytic in
   $\circlevar$ and the integral vanishes.  Now the integrand of each term of
   (\ref{abc}) is clearly analytic in $s$, hence so is $g_\indexvar$ itself,
   establishing item \ref{g_m-analytic}.  

   For item \ref{g_m-size}, let $U_0 := \sup_{\abs{\circlevar}=\radius, \abs{s}
     \le s_0} \abs{U(\circlevar,s)}$, and $V_0 := \sup_{\abs{\circlevar}=\radius}
   \abs{V(\circlevar)}$.  Then for $\abs{s} \le s_0$,
   \begin{align*}
     \abs{g_\indexvar(s)} &\le \sum_{\substack{a+b+c=\indexvar\\a-b-n \le -1}} \binom{\indexvar}{a,b,c}
     4^{-(a+b)} (2 \pi \radius) 
     \radius^{a-b-n} s_0^{-(a-b-n)-1} \left(s_0 U_0\right)^c V_0
     \\
     &\le 2 \pi \radius V_0 \frac{s_0^{n-1}}{\radius^n} \sum_{\substack{a+b+c=\indexvar\\a-b-n \le -1}}
     \binom{\indexvar}{a,b,c} \left( \frac{\radius}{4 s_0}\right)^a
     \left(\frac{s_0}{4 \radius}\right)^b \left(s_0 U_0\right)^c
     \\
     &\le 2 \pi \radius V_0 \frac{s_0^{n-1}}{\radius^n} \sum_{a+b+c=\indexvar}
     \binom{\indexvar}{a,b,c} \left( \frac{\radius}{4 s_0}\right)^a
     \left(\frac{s_0}{4 \radius}\right)^b \left(s_0 U_0\right)^c
     \\
     &\le 2 \pi \radius V_0 \frac{s_0^{n-1}}{\radius^n} \left(\frac{\radius}{4
       s_0} + \frac{s_0}{4 \radius} + s_0 U_0\right)^\indexvar
   \end{align*}
   so that a constant $C$ can be chosen with $g_\indexvar(s) \le C^\indexvar$,
   establishing item \ref{g_m-size}.

   For item \ref{g_m-zero-low}, suppose $\indexvar \le n-1$ and let $h_\indexvar(s) =
   s^{\indexvar-n+1} g_\indexvar(s)$, so that
   \begin{align*}
     h_\indexvar(s) = \sum_{\substack{a+b+c=\indexvar\\a-b-n \le -1}} \binom{\indexvar}{a,b,c} 4^{-(a+b)} \oint_{\partial B(0,\radius)}
     \circlevar^{a-b-n} s^{-(a-b-\indexvar)} \left(sU(\circlevar,s)\right)^c V(\circlevar) (-i) \,d\circlevar
   \end{align*}
   But $a-b-\indexvar \le a-\indexvar \le 0$ since $a \le \indexvar$ by definition, so only
   positive powers of $s$ appear, and $h_\indexvar$ is analytic in $s$.

   For item \ref{g_m-zero-high}, we see that when $s=0$, each term of
   (\ref{abc}) will vanish unless $c=0$ and $a-b-n=-1$, i.e. $a+b=\indexvar$
   and $a-b=n-1$.  If $\indexvar$ and $n$ have the same parity, this happens
   for no term, so $g_\indexvar(0)=0$.  If $\indexvar$ and $n$ have opposite parity,
   this forces $a=(\indexvar+n-1)/2$, $b=(\indexvar-n+1)/2$, both of which are
   nonnegative integers.  In this case
   \begin{align*}
     g_\indexvar(s) &= \binom{\indexvar}{(\indexvar+n-1)/2}
     4^{-\indexvar} \oint \circlevar^{-1} V(\circlevar)
     (-i)\,d\circlevar \\
     &= \binom{\indexvar}{(\indexvar+n-1)/2}
     4^{-\indexvar} 2 \pi V(0) > 0
   \end{align*}
   since $V(0)>0$.
\end{proof}

 From this we derive corresponding properties of the function $F$.

 \begin{corollary}\label{Fstructure}
   Let $F(\oldlambda,s)$ be defined as in (\ref{Fdef}).  Then for all $s_0
   < \pi$:
   \begin{enumerate}
     \item $F$ is analytic for all $\oldlambda$ and all $0 \le s \le s_0$.
     \item We may write
       \begin{equation}
         F(\oldlambda,s) = e^{-\oldlambda/2} \left[ \sum_{\indexvar=0}^{n-1} \frac{\oldlambda^{\indexvar}
           s^{n-1-\indexvar}}{\indexvar!} h_\indexvar(s) + \oldlambda^n H(\oldlambda,s) \right]
       \end{equation}
       with $h_\indexvar, H$ analytic for all $\oldlambda$ and all $0 \le s \le s_0$.
       Furthermore, $h_{n-1}(0) > 0$
     \item $F(\oldlambda, 0) > 0$ for all $\oldlambda > 0$.
   \end{enumerate}
 \end{corollary}

 \begin{proof}
 We prove the corresponding facts about $G=e^{\oldlambda/2}F$.  By items
 \ref{g_m-analytic} and \ref{g_m-size} of Lemma \ref{g_m}, we have
 that $G$ is analytic for $\abs{s} \le s_0$ and all $\oldlambda$, since
 the sum in (\ref{Gsum}) is a sum of analytic functions and converges
 uniformly.  By item \ref{g_m-zero-low} we have that
 \begin{equation*}
   G(\oldlambda, s) = \sum_{\indexvar=0}^{n-1} \frac{\oldlambda^\indexvar s^{n-1-\indexvar}}{\indexvar!}
   h_\indexvar(s) + \oldlambda^n \sum_{\indexvar=0}^\infty \frac{\oldlambda^\indexvar}{(n+\indexvar)!} g_{n+\indexvar}(s).
 \end{equation*}
 And by items \ref{g_m-zero-low} and \ref{g_m-zero-high}, $G(\oldlambda,0)
 = \sum_{\indexvar=n-1}^\infty \frac{\oldlambda^\indexvar g_\indexvar(0)}{\indexvar!} > 0$ for all
 $\oldlambda > 0$.
 \end{proof}

 \begin{proposition}\label{Fprop-II}
   For all $\oldlambda_1 > 0$, there exist $\delta > 0$, and $0 <
   C_1' \le C_2' < \infty$ such that
   \begin{equation}\label{Fprop-II-eqn}
     C_1' \oldlambda^{n-1} \le \Re F(\oldlambda,s) \le \abs{F(\oldlambda,s)} \le
     C_2' \oldlambda^{n-1}
   \end{equation}
   for all $0 \le \oldlambda < \oldlambda_1$, $0 \le s < \delta
   \oldlambda$.  (Here we are treating $\oldlambda$ and $s$ as real variables.)
 \end{proposition}

 \begin{proof}
    Let $K$ be a positive constant so large that $\abs{h_\indexvar(s)}
    \le K$ and $\abs{H(\oldlambda,s)} \le K$ for all $0 \le \oldlambda
    < \oldlambda_1$, $0 \le s < \oldlambda_1$, $\indexvar \le n-1$.
    For any $\delta < 1$ and all $s \le \delta \oldlambda < \oldlambda_1$, we
    have
   \begin{align*}
     \Re G(\oldlambda, s) &= \frac{\oldlambda^{n-1}}{(n-1)!} \Re
     h_{n-1}(s) + 
            \sum_{\indexvar=0}^{n-2} \frac{\oldlambda^{\indexvar} s^{n-1-\indexvar}}{\indexvar!}
           \Re {h_\indexvar(s)} + \oldlambda^n \Re H(\oldlambda,s)
           \\ &\ge \frac{\oldlambda^{n-1}}{(n-1)!} \Re h_{n-1}(s) -
           \sum_{\indexvar=0}^{n-2} \frac{\oldlambda^{n-1} \delta^{n-1-\indexvar} K}{\indexvar!}
           - \oldlambda^n K \\
           &= \oldlambda^{n-1} \left[ \frac{\Re h_{n-1}(s)}{(n-1)!} - K \sum_{\indexvar=0}^{n-2} \frac{\delta^{n-1-\indexvar}}{\indexvar!}
           \right] - \oldlambda^n K.
   \end{align*}
   Since $h_{n-1}(0) > 0$, we may now choose $\delta$ so small that the bracketed term is positive
   for all $0 \le s \le \delta \oldlambda_1$.  Then there exists $\oldlambda_0 >
   0$ so small that for all $0 \le \oldlambda \le \oldlambda_0$, we have $\Re
   F(\oldlambda,s) \ge e^{-\oldlambda_0/2} \Re G(\oldlambda,s) \ge C_1'
   \oldlambda^{n-1}$ for some $C_1' > 0$.  On the other hand,
   \begin{align*}
     \abs{F(\oldlambda, s)} &\le \abs{G(\oldlambda,s)} \\
     &\le \sum_{\indexvar=0}^{n-1} \frac{\oldlambda^{\indexvar} s^{n-1-\indexvar}}{\indexvar!}
           \abs{h_\indexvar(s)} + \oldlambda^n \Re H(\oldlambda,s) \\
           &\le \oldlambda^{n-1} \sum_{\indexvar=0}^{n-1} \frac{K \delta^{n-1-\indexvar}}{\indexvar!}
            + \oldlambda^n K.
   \end{align*}
   Again, for small $\oldlambda$ (take $\oldlambda_0$ smaller if necessary),
   we have $\abs{F(\oldlambda,s)} \le C_2' \oldlambda^{n-1}$.

   It remains to handle $\oldlambda_0 \le \oldlambda \le \oldlambda_1$.  But
   this presents no difficulty; as $F(\oldlambda,0) > 0$ for all $\oldlambda
   > 0$, and $F$ is continuous, there exists $\delta$ so small that
   \begin{equation*}
     \inf_{\oldlambda_0 \le \oldlambda \le \oldlambda_1, 0 \le s \le \delta
     \oldlambda_1} \Re F(\oldlambda, s) > 0.
   \end{equation*}
   This completes the proof.
 \end{proof}

 \begin{proposition}\label{Fprop-III}
   There exists $\oldlambda_1 > 0$, $s_0 > 0$ and constants $C_1, C_2
   > 0$ such that
   \begin{equation}\label{Fprop-III-eqn}
     \frac{C_1}{\sqrt{\oldlambda}} \le \Re F(\oldlambda,s) \le
     \abs{F(\oldlambda,s)} \le \frac{C_2}{\sqrt{\oldlambda}}
   \end{equation}
   for all $\oldlambda > \oldlambda_1$, $0 < s < s_0$.
 \end{proposition}

 \begin{proof}
   Here the Gaussian approximation technique of Section
   \ref{steepest-descent-sec} is again applicable.  We will fix the
   contour in (\ref{Fdef}) as a circle of radius $\radius = s$,
   parametrize it, and examine the integrand directly.  Thus let $w =
   se^{i\polar}$ in (\ref{Fdef}) to obtain
 \begin{align}
   F(\oldlambda,s) = s^{n-1} \int_{-\pi}^\pi e^{-\oldlambda \phi(se^{i\polar},s)}
   a(i(\pi-se^{i\polar}))se^{i\polar}\,d\polar.
 \end{align}
 We shall apply Lemma \ref{abstract-lemma-multi}, with $m=1$, $\lambda=\polar$,
 $r=\pi$, $x=\sqrt{\oldlambda}$.  Let 
 \begin{align}
 g(\polar, s) &= \Re \phi(se^{i\polar},s) \\
 k(\sqrt{\oldlambda}, \polar, s) &= e^{-i \sqrt{\oldlambda}^2 \Im
   \phi(se^{i\polar},s)} s^n a(i(\pi-se^{i\polar}))e^{i\polar}
 \end{align}

 Since $\phi(s,s) = 0$ and $\circlevar=s$ is a critical point of
 $\phi(\circlevar,s)$, we have 
 \begin{align}
   \evalat{\frac{\partial^2}{\partial^2\polar} \phi(se^{i\polar},
   s)}{\polar=0} &= \frac{s}{4\pi} \phi''(s,s) (is)^2 = \frac{s^3
     \nu'(\pi-s)}{4\pi}
 \end{align}
which is bounded and positive for all small $s$ (recall $\nu(\pi-s)
\sim s^{-2}$).  Thus there exists $s_0, \epsilon$ small enough and
constants $b_1, b_2$ such that
\begin{equation}\label{g-bounds}
  b_1 \polar^2 \le g(\polar,s) \le b_2 \polar^2
\end{equation}
for $s < s_0$, $\abs{\polar} < \epsilon$.  Also, we have from (\ref{phi-series}) that
\begin{equation}
  \phi(se^{i\polar},s) = \frac{1}{2} - \frac{1}{2}\cos\polar -sU(se^{i\polar},s)
\end{equation}
so that by taking $s_0$ smaller if necessary, we can ensure
$g(\polar,s) > 0$ for all $s < s_0$ and $\epsilon \le \abs{\polar} \le
\pi$.  Thus (\ref{g-bounds}) holds for $s < s_0$ and all $\polar \in
   [-\pi,\pi]$, with possibly different constants $b_1, b_2$.

Boundedness of $k$ follows from the fact that $a$ has a pole of order
$n$ at $i\pi$, so $s^n a(i(\pi-se^{i\polar})) = V(se^{i\polar})$ is
bounded for small $s$.  Finally, since
$\evalat{\frac{\partial^2}{\partial^2\polar} \phi(se^{i\polar},
  s)}{\polar=0} > 0$ and $V(0)>0$, the argument used in the proof of Theorem
\ref{region-I-theorem} shows that the necessary lower bound on $k$
also holds.  Then an application of Lemma \ref{abstract-lemma-multi}
completes the proof.
 \end{proof}

 \begin{proof}[Proof of Theorem \ref{h-estimate}]
   Choose $\oldlambda_1, s_0$ so that Proposition \ref{Fprop-III} holds, and take
   $B_1$ large enough so that $\theta(x,z) \ge \pi-s$ when $\abs{z}
   \ge B_1 \abs{x}^2$.  Use this value of $\oldlambda_1$ and choose a
   $\delta$ such that Proposition \ref{Fprop-II} holds, and take $d_0$
   large enough that $s < \delta \oldlambda$ when $\abs{z} \ge d_0$
   (see (\ref{z-comparisons})).  So for such $(x,z)$, either
   (\ref{Fprop-II-eqn}) or (\ref{Fprop-III-eqn}) holds; which one depends on the value
   of $\oldlambda = \oldlambda (x,z)$.  We can combine them to get
   \begin{equation}\label{F-unified}
     C_1' \frac{\oldlambda^{n-1}}{1+\oldlambda^{n-\frac{1}{2}}} \le \Re
     F(\oldlambda,s) \le \abs{F(\oldlambda,s)} \le C_2'
     \frac{\oldlambda^{n-1}}{1+\oldlambda^{n-\frac{1}{2}}}.
   \end{equation}
   Inserting this into (\ref{hr-F}) and using (\ref{z-comparisons}),
   we have (in more compact notation)
   \begin{equation}
     h_r(x,z) \asymp \left(\frac{\oldlambda}{s}\right)^{n-1}
     \frac{1}{1+\oldlambda^{n-\frac{1}{2}}} \asymp
     \frac{\abs{z}^{n-1}}{1+(\abs{x}\sqrt{\abs{z}})^{n-\frac{1}{2}}}.
   \end{equation}
   By Lemma (\ref{hl-estimate}), $h_l$ is clearly negligible by
   comparison, so Theorem \ref{h-estimate} is proved.
\end{proof}

A similar argument will give us the estimates on $\grad p_1$ and $q_2$
which correspond to Theorems \ref{main-gradient-theorem} and
\ref{vertical-gradient-theorem}.

\begin{theorem}\label{gradient-thm-regionII}
  For $m$ odd, there exist constants $B_1, d_0, C$ such that
  \begin{equation}
    \abs{\grad p_1(x,z)} \asymp \frac{\abs{x}
      d(x,z)^{2n-m+1}}{1+(\abs{x}d(x,z))^{n+\frac{1}{2}}}e^{-\frac{1}{4}d(x,z)^2}
  \end{equation}
and
\begin{equation} \label{q2-upper-regionII}
  \abs{q_2(x,z)} \le C
  \frac{d(x,z)^{2n-m-1}}{1+(\abs{x}d(x,z))^{n-\frac{1}{2}}}e^{-\frac{1}{4}d(x,z)^2}
\end{equation}
  whenever $\abs{z}\ge B_1\abs{x}^2$ and $d(x,z) \ge d_0$.
\end{theorem}

\begin{proof}
  Applying (\ref{grad-q1-q2}) to (\ref{p-expansion}), we have
\begin{align*}
  \grad p_1(x,z) &= -\frac{1}{2}(2\pi)^{-m}(4\pi)^{-n}\abs{x} (q_1(x,z) \hat{x} + q_2(x,z)
  J_{\hat{z}}\hat{x}) \label{grad-q1-q2}\\
   \intertext{where}
   q_1(x,z) &= -\frac{2}{\abs{x}} \frac{\partial p_1(x,z)}{\partial
   \abs{x}} \\
   &= - \sum_{k=0}^{(m-1)/2} c_{m,k}  \abs{z}^{k-m+1} \int_{-\infty}^\infty
   e^{{i} \rho \abs{z} -\frac{\abs{x}^2}{4} \rho \coth
     \rho} \left(\frac{\rho}{\sinh\rho}\right)^{n+1} (-\cosh\rho) (-i\rho)^{k}
  \,d\rho \\
 q_2(x,z) &=\frac{\partial p_1(x,z)}{\partial
   \abs{z}} \\
 &= \sum_{k=0}^{(m-1)/2} \left[ c_{m,k} (k-m+1)\abs{z}^{k-m} \int_{-\infty}^\infty
   e^{{i} \rho \abs{z} -\frac{\abs{x}^2}{4} \rho \coth
     \rho} \left(\frac{\rho}{\sinh\rho}\right)^n (-i\rho)^{k}
   \,d\rho\right] \\
   &\quad  - \sum_{k=0}^{(m-1)/2} \left[c_{m,k} \abs{z}^{k-m+1} \int_{-\infty}^\infty
   e^{{i} \rho \abs{z} -\frac{\abs{x}^2}{4} \rho \coth
     \rho} \left(\frac{\rho}{\sinh\rho}\right)^n (-i\rho)^{k+1}
   \,d\rho \right]
\end{align*}
  Each integral can be estimated by Theorem \ref{h-estimate}.  For
  $q_1$, each integral is comparable to
  $e^{-\frac{1}{4}d(x,z)^2}\frac{\abs{z}^{n}}{1+(\abs{x}\sqrt{\abs{z}})^{n+\frac{1}{2}}}$,
  and the \mbox{$k=(m-1)/2$} term dominates, so
  \begin{equation}\label{q1-estimate-regionII}
    \abs{q_1(x,z)} \asymp \frac{\abs{z}^{n-(m-1)/2}}{1+(\abs{x}\sqrt{\abs{z}})^{n+\frac{1}{2}}}e^{-\frac{1}{4}d(x,z)^2}.
  \end{equation}
  The appearance of the extra minus sign in $q_1$ is to account for
  the fact that $\cosh(i\pi) = -1$, but Theorem \ref{h-estimate}
  requires that $a(\lambda)$ be positive near $\lambda=i\pi$.

  For $q_2$, each integral is comparable to
  $\frac{\abs{z}^{n-1}}{1+(\abs{x}\sqrt{\abs{z}})^{n-\frac{1}{2}}}e^{-\frac{1}{4}d(x,z)^2}$,
  and the \mbox{$k=(m-1)/2$} term of the second sum dominates, so
  \begin{equation}\label{q2-estimate-regionII}
    \abs{q_2(x,z)} \asymp
    \frac{\abs{z}^{n-1-(m-1)/2}}{1+(\abs{x}\sqrt{\abs{z}})^{n-\frac{1}{2}}}
    e^{-\frac{1}{4}d(x,z)^2}
  \end{equation}
  which in particular implies (\ref{q2-upper-regionII}).  To combine
  (\ref{q1-estimate-regionII}) and (\ref{q2-estimate-regionII}), note
  that for $\abs{x}^2 \abs{z}$ bounded we have
  \begin{equation}
    \abs{q_1(x,z)} \asymp \abs{z}^{n-(m-1)/2}e^{-\frac{1}{4}d(x,z)^2}; \quad
    \abs{q_2(x,z)} \asymp \abs{z}^{n-1-(m-1)/2}e^{-\frac{1}{4}d(x,z)^2}
  \end{equation}
  so that the $q_1$ term dominates, and
  \begin{equation}
    \abs{\grad p_1(x,z)} \asymp \abs{x} \abs{z}^{n-(m-1)/2}e^{-\frac{1}{4}d(x,z)^2}.
  \end{equation}
  For $\abs{x}^2 \abs{z}$ bounded away from $0$ we have
  \begin{equation}
    \begin{split}
    \abs{q_1(x,z)} &\asymp \abs{x}^{-n-\frac{1}{2}} \abs{z}^{\frac{n}{2}
      - \frac{m}{2} + \frac{1}{4}}e^{-\frac{1}{4}d(x,z)^2} 
\\  \abs{q_2(x,z)} &\asymp
    \abs{x}^{-n+\frac{1}{2}} \abs{z}^{\frac{n}{2} - \frac{m}{2} -
      \frac{1}{4}}e^{-\frac{1}{4}d(x,z)^2} \asymp
    \frac{\abs{x}}{\sqrt{\abs{z}}} q_1(x,z)
    \end{split}
  \end{equation}
  so that the $q_1$ term dominates again
  ($\frac{\abs{x}}{\sqrt{\abs{z}}}$ is bounded by assumption).  Thus
  \begin{equation}
    \abs{\grad p_1(x,z)} \asymp \abs{x}
    \frac{\abs{z}^{n-(m-1)/2}}{1+(\abs{x}\sqrt{\abs{z}})^{n+\frac{1}{2}}}e^{-\frac{1}{4}d(x,z)^2}
  \end{equation}
  which is equivalent to the desired estimate.
\end{proof}

\section{Hadamard descent}\label{hadamard-sec}

In this section, we obtain estimates for $p_1(x,z)$ and $\abs{\grad
  p_1(x,z)}$ for $\abs{z} \ge B_1 \abs{x}^2$, $\abs{z} \ge d_0$, in
the case where the center dimension $m$ is even.  The methods of the
previous section are not directly applicable, but we can deduce an
estimate for even $m$ by integrating the corresponding estimate for
$m+1$.  As discussed in the remark at the end of Section
\ref{sublaplacian-sec}, this is valid even though there may not exist
an $H$-type group of dimension $2n+m+1$ with center dimension $m+1$,
since the estimates we use are derived from the formula
(\ref{Rm-integral}) and hold for all values of $n,m$.

We continue to assume that $\abs{z} \ge B_1 \abs{x}^2$ and $\abs{z}
\ge d_0$ for some sufficiently large $B_1, d_0$.  To emphasize the
dependence on the dimension, we write $p^{(n,m)}$ for the function
$p_1$ in (\ref{Rm-integral}).

In order to estimate $p^{(n,m)}$ for $m$ even, we consider
$p^{(n,m+1)}$.  We can observe that
\begin{equation}\label{dimension-reduce}
  p^{(n,m)}(x,z) = \int_\R p^{(n,m+1)}(x, (z, z_{m+1}))\,dz_{m+1}
\end{equation}
since $\int_\R\int_\R e^{i \lambda_{m+1} z_{m+1}}
f(\lambda_{m+1})\,d\lambda_{m+1}\,dz_{m+1} = 2 \pi f(0)$.  Note that
$\abs{(\lambda,0)}_{\R^{m+1}} = \abs{\lambda}_{\R^m}$.  Now
$p^{(n,m+1)}$ can be estimated by means of Theorem
\ref{region-II-III-theorem}.  Using the fact that $\abs{(z,z_{m+1})}
\ge \abs{z}$, we have that for $m$ even, there exist constants $B_1,
d_0$ such that
  \begin{equation}\label{p-Q-comp}
    p^{(n,m)}(x,z) \asymp  Q^{(2n-m-2,n-\frac{1}{2})}(x,z)
  \end{equation}
  whenever $\abs{z} \ge B_1 \abs{x}^2$ and $\abs{z} \ge d_0$, where
  \begin{equation}\label{Qnm-def}
    Q^{(\alpha,\beta)}(x,z) := \int_\R \frac{d(x,(z,z_{m+1}))^{\alpha}}{1+(\abs{x}d(x,(z,z_{m+1})))^{\beta}}
    e^{-\frac{1}{4}d(x,(z,z_{m+1}))^2}\,dz_{m+1}
  \end{equation}
  Thus it suffices to estimate the integrated bounds given by
  $Q^{(\alpha,\beta)}$.

  \begin{lemma}\label{Q-est}
    For $\abs{z} \ge B_1 \abs{x}^2$ and $\abs{z} \ge d_0$, we have
    \begin{equation}
      Q^{(\alpha,\beta)}(x,z) \asymp \frac{d(x,z)^{\alpha +
          1}}{1+(\abs{x}d(x,z))^\beta} e^{-\frac{1}{4}d(x,z)^2}.
    \end{equation}
  \end{lemma}

We will require two preliminary computations.  Since $d(x,z)$ depends
on $z$ only through $\abs{z}$, we will occasionally treat $d$ as a
function on $\R^{2n} \times [0,\infty)$.

\begin{lemma}\label{dist-z-deriv}
  There exist positive constants $c_1, c_2, B_1$ such that for all
  $x \in \R^{2n}, u \in \R$ with $u
  \ge B_1 \abs{x}^2$, we have $0 < c_1 \le \frac{\partial}{\partial u}
  d(x,u)^2 \le c_2 < \infty$.
\end{lemma}

\begin{proof}
  Let $\mu(\theta) = \frac{\theta^2}{\sin^2\theta}$, so that $d(x,u)^2
  = \abs{x}^2 \mu(\theta)$ with $\theta = \theta(x,z) =
  \nu^{-1}\left(\frac{2u}{\abs{x}^2}\right)$.  Then
  \begin{equation}
    \frac{\partial}{\partial u}
  d(x,u)^2 = 2 \frac{\mu'(\theta)}{\nu'(\theta)}.
  \end{equation}
  It is easily verified that $\mu'(\theta) > 0$, $\nu'(\theta) > 0$
  for all $\theta \in (0,\pi)$, and $\frac{\mu'(\theta)}{\nu'(\theta)}
  \to \pi > 0$ as $\theta \to \pi$.
\end{proof}

\begin{lemma}\label{exp-poly-est}
  For any $\alpha \in \R$, there exists $C_\alpha > 0$ such that for
  all $w_0 \ge 1$ we have
  \begin{equation}
    \int_{w_0}^\infty w^\alpha e^{-w}\,dw \le C_\alpha w_0^\alpha e^{-w_0}.
  \end{equation}
\end{lemma}

\begin{proof}
  For $\alpha \le 0$, $w^\alpha$ is decreasing for $w \ge 1$, so
  \begin{equation}
    \int_{w_0}^\infty w^\alpha e^{-w}\,dw \le w_0^\alpha
    \int_{w_0}^\infty e^{-w}\,dw = w_0^\alpha e^{-w_0}
  \end{equation}
  and this holds with $C_\alpha = 1$.  Now, for a nonnegative integer
  $n$, suppose the lemma holds for all $\alpha \le n$.  Then if $n <
  \alpha \le n+1$, we integrate by parts to obtain
  \begin{align*}
    \int_{w_0}^\infty w^\alpha e^{-w}\,dw = w_0^\alpha e^{-w_0} +
    \alpha \int_{w_0}^{\infty} w^{\alpha - 1} e^{-w}\,dw \le (1+\alpha
    C_{\alpha - 1}) w_0^\alpha e^{-w_0}
  \end{align*}
  so that the lemma also holds for all $\alpha \le n+1$.  By induction
  the proof is complete.
\end{proof}

\begin{proof}[Proof of Lemma \ref{Q-est}]
We make the change of variables $u = \abs{(z, z_{m+1})}$ so that
$z_{m+1} = \sqrt{u^2 - \abs{z}^2}$.  By our previous abuse of
notation, we can write $d(x,(z,z_{m+1}))=d(x,u)$.  Thus
\begin{align*}
  Q^{(\alpha,\beta)}(x,z) &= \int_{\abs{z}}^{\infty} 
  \frac{d(x,u)^{\alpha}}{1+(\abs{x}d(x,u))^{\beta}}
  e^{-\frac{1}{4}d(x,u)^2} \frac{u}{\sqrt{u^2-\abs{z}^2}}\,du \\
    &\asymp \int_{\abs{z}}^{\infty} \frac{1}{\sqrt{u-\abs{z}}}
  \frac{1}{\sqrt{u+\abs{z}}}
  \frac{d(x,u)^{\alpha+2}}{1+(\abs{x}d(x,u))^{\beta}} e^{-\frac{1}{4}d(x,u)^2} \,du.
\end{align*}
We used the fact that $u \asymp d(x,u)^2$ where $\abs{z} \ge B_1
\abs{x}^2$, by Corollary \ref{distance-estimate}.

Now, noting that $u \mapsto d(x,u)$ is an increasing function, and $w
\mapsto w^{\alpha+2}e^{-\frac{1}{4}w^2}$ is decreasing for large
enough $w$, the lower bound can be obtained by
\begin{align*}
  Q^{(\alpha,\beta)}(x,z) &\ge \int_{\abs{z}}^{\abs{z}+1} \frac{1}{\sqrt{u-\abs{z}}}
  \frac{1}{\sqrt{u+\abs{z}}}
  \frac{d(x,u)^{\alpha+2}}{1+(\abs{x}d(x,u))^{\beta}}
  e^{-\frac{1}{4}d(x,u)^2} \,du  \\
&\ge \left( \int_{\abs{z}}^{\abs{z}+1} \frac{1}{\sqrt{u-\abs{z}}}
  \,du  \right) \frac{1}{\sqrt{2\abs{z}+1}}
  \frac{d(x,\abs{z}+1)^{\alpha+2}}{1+(\abs{x}d(x,\abs{z}+1))^{\beta}}
  e^{-\frac{1}{4}d(x,\abs{z}+1)^2}  \\
  &= 2 \frac{1}{\sqrt{2\abs{z}+1}}
  \frac{d(x,\abs{z}+1)^{\alpha+2}}{1+(\abs{x}d(x,\abs{z}+1))^{\beta}}
  e^{-\frac{1}{4}d(x,\abs{z}+1)^2} \\
  &\ge C \frac{1}{\sqrt{2\abs{z}}}
  \frac{d(x,{z})^{\alpha+2}}{1+(\abs{x}d(x,{z}))^{\beta}}
  e^{-\frac{1}{4}d(x,{z})^2}
\end{align*}
where the last line follows because $u \mapsto d(x,u)^2$ is Lipschitz,
as shown by Lemma \ref{dist-z-deriv}, with a constant independent of $x$.

Since $\abs{z} \asymp d(x,z)^2$, we have that
\begin{equation}
  Q^{(\alpha,\beta)}(x,z)\ge C'
  \frac{d(x,{z})^{\alpha+1}}{1+(\abs{x}d(x,{z}))^{\beta}}
  e^{-\frac{1}{4}d(x,{z})^2}.
\end{equation}

For an upper bound, we have
\begin{align*}
  Q^{(\alpha,\beta)}(x,z) \le C \left[ \int_{\abs{z}}^{\abs{z}+1} \frac{1}{\sqrt{u-\abs{z}}}
  \frac{1}{\sqrt{u+\abs{z}}}
  \frac{d(x,u)^{\alpha+2}}{1+(\abs{x}d(x,u))^{\beta}}
  e^{-\frac{1}{4}d(x,u)^2} \,du + \int_{\abs{z}+1}^\infty \dots. \right]
\end{align*}
Now
\begin{align*}
  &\quad \int_{\abs{z}}^{\abs{z}+1} \frac{1}{\sqrt{u-\abs{z}}}
  \frac{1}{\sqrt{u+\abs{z}}}
  \frac{d(x,u)^{\alpha+2}}{1+(\abs{x}d(x,u))^{\beta}}
  e^{-\frac{1}{4}d(x,u)^2} \,du \\
&\le \left(\int_{\abs{z}}^{\abs{z}+1} \frac{1}{\sqrt{u-\abs{z}}}
   \,du\right) \frac{1}{\sqrt{2\abs{z}}}
  \frac{d(x,z)^{\alpha+2}}{1+(\abs{x}d(x,z))^{\beta}}
  e^{-\frac{1}{4}d(x,z)^2}  \\
   &= 2 \frac{1}{\sqrt{2\abs{z}}}
  \frac{d(x,z)^{\alpha+2}}{1+(\abs{x}d(x,z))^{\beta}}
  e^{-\frac{1}{4}d(x,z)^2} \\
  &\le C \frac{d(x,z)^{\alpha+1}}{1+(\abs{x}d(x,z))^{\beta}}
  e^{-\frac{1}{4}d(x,z)^2}.
\end{align*}
For the other term, we observe
\begin{align*}
  \int_{\abs{z}+1}^{\infty} \frac{1}{\sqrt{u-\abs{z}}}
  \frac{1}{\sqrt{u+\abs{z}}}
  \frac{d(x,u)^{\alpha+2}}{1+(\abs{x}d(x,u))^{\beta}}
  e^{-\frac{1}{4}d(x,u)^2} \,du &\le \int_{\abs{z}+1}^{\infty}   \frac{1}{\sqrt{u+\abs{z}}}
  \frac{d(x,u)^{\alpha+2}}{1+(\abs{x}d(x,u))^{\beta}}
  e^{-\frac{1}{4}d(x,u)^2} \,du \\
  &\le \int_{\abs{z}}^{\infty}   \frac{1}{\sqrt{2u}}
  \frac{d(x,u)^{\alpha+2}}{1+(\abs{x}d(x,u))^{\beta}}
  e^{-\frac{1}{4}d(x,u)^2} \,du \\
  &\le C \int_{\abs{z}}^{\infty}   \frac{d(x,u)^{\alpha+1}}{1+(\abs{x}d(x,u))^{\beta}}
  e^{-\frac{1}{4}d(x,u)^2} \,du
\end{align*}
We now make the change of variables $w=\frac{1}{4}d(x,u)^2$.  By the
above lemma, $du/dw$ is bounded, so
\begin{align*}
  \int_{\abs{z}}^{\infty}   \frac{d(x,u)^{\alpha+1}}{1+(\abs{x}d(x,u))^{\beta}}
  e^{-\frac{1}{4}d(x,u)^2} \,du \le C \int_{\frac{1}{4}d(x,z)^2}^\infty
  \frac{(4w)^{(\alpha+1)/2}}{1+(2\abs{x}\sqrt{w})^\beta} e^{-w}\,dw.
\end{align*}
If $d(x,z) \le 1/\abs{x}$, we have
\begin{align*}
  \int_{\frac{1}{4}d(x,z)^2}^\infty
  \frac{(4w)^{(\alpha+1)/2}}{1+(2\abs{x}\sqrt{w})^\beta} e^{-w}\,dw
  &\le \int_{\frac{1}{4}d(x,z)^2}^\infty
  {(4w)^{(\alpha+1)/2}} e^{-w}\,dw \\
  &\le C d(x,z)^{\alpha+1} e^{-\frac{1}{4}d(x,z)^2} \\
  &\le 2C \frac{d(x,z)^{\alpha+1}}{1+(\abs{x}d(x,z))^\beta}e^{-\frac{1}{4}d(x,z)^2}
\end{align*}
where we have used Lemma \ref{exp-poly-est}.

On the other hand, when $d(x,z) \ge 1/\abs{x}$, we have
\begin{align*}
  \int_{\frac{1}{4}d(x,z)^2}^\infty
  \frac{(4w)^{(\alpha+1)/2}}{1+(2\abs{x}\sqrt{w})^\beta} e^{-w}\,dw
  &\le (2\abs{x})^{-\beta} \int_{\frac{1}{4}d(x,z)^2}^\infty
  {(4w)^{(\alpha+1-\beta)/2}} e^{-w}\,dw \\
  &\le C \abs{x}^{-\beta} d(x,z)^{\alpha+1-\beta} e^{-\frac{1}{4}d(x,z)^2} \\
  &\le 2C \frac{d(x,z)^{\alpha+1}}{1+(\abs{x}d(x,z))^\beta}e^{-\frac{1}{4}d(x,z)^2}
\end{align*}

Combining all this, we have as desired that
\begin{equation}
  Q^{(\alpha,\beta)}(x,z) \asymp \frac{d(x,z)^{\alpha+1}}{1+(\abs{x}d(x,z))^\beta}e^{-\frac{1}{4}d(x,z)^2}.
\end{equation}
\end{proof}

\begin{corollary}
  Theorems  \ref{region-II-III-theorem} and
  \ref{gradient-thm-regionII} also hold for $m$ even.
\end{corollary}

\begin{proof}
The heat kernel estimate of Theorem \ref{region-II-III-theorem} is
immediate, given (\ref{p-Q-comp}) and Lemma \ref{Q-est}.
  
To obtain an estimate on $\grad p_1$, we define $q_1^{(n,m)} :=
-\frac{2}{\abs{x}} \frac{\partial}{\partial \abs{x}} p_1^{(n,m)}(x,z)$,
$q_2^{(n,m)} := \frac{\partial}{\partial \abs{z}} p_1^{(n,m)}(x,z)$,
as in (\ref{grad-q1-q2}).

For $q_1$, we simply differentiate (\ref{dimension-reduce}) to see
\begin{align*}
  q_1^{(n,m)}(x,z) &= \int_\R q_1^{(n,m+1)}(x,(z,z_{m+1}))\,dz_{m+1}
  \\
  &\asymp Q^{(2n-m,n+\frac{1}{2})}(x,z) && \text{by
    (\ref{q1-estimate-regionII})}\\
  &\asymp
  \frac{d(x,z)^{2n-m+1}}{1+(\abs{x}d(x,z))^{n+\frac{1}{2}}}e^{-\frac{1}{4}d(x,z)^2}
  && \text{by Lemma \ref{Q-est}.}
\end{align*}

For $q_2$, we again differentiate (\ref{dimension-reduce}).  Here we
obtain
\begin{align*}
  q_2^{(n,m)}(x,z) &= \int_\R q_2^{(n,m+1)}(x,(z,z_{m+1}))
  \frac{\abs{z}}{\abs{(z,z_{m+1})}}\,dz_{m+1} \\
  &\asymp \abs{z} Q^{(2n-m-4,n-\frac{1}{2})} && \text{by
    (\ref{q2-estimate-regionII})} \\
  &\asymp d(x,z)^2
  \frac{d(x,z)^{2n-m-3}}{1+(\abs{x}d(x,z))^{n-\frac{1}{2}}}e^{-\frac{1}{4}d(x,z)^2}
  \\
  &\asymp
  \frac{d(x,z)^{2n-m-1}}{1+(\abs{x}d(x,z))^{n-\frac{1}{2}}}e^{-\frac{1}{4}d(x,z)^2}.
\end{align*}

Repeating the computation from Theorem \ref{gradient-thm-regionII}, we
have the desired estimates on $\abs{\grad p_1}$ and $\abs{q_2}$.
\end{proof}

\section{Conclusion}

An obvious extension of this result would be to obtain precise
estimates for the heat kernel in more general nilpotent Lie groups.
For step-2 nilpotent groups, a formula for the heat kernel along the
lines of (\ref{Rm-integral}) can be found in \cite{cygan}, among
others.  However, the additional algebraic structure enjoyed by H-type
groups has played a major part in the analysis presented here, and its
absence complicates matters considerably.  A key difficulty is that
the exponent in the formula for $p_t$ now contains expressions like
$J_\lambda \cot J_\lambda$, which are awkward to work with when
$J_\lambda$ may not commute with its derivatives with respect to $\lambda$.

\def\cprime{$'$}


\begin{thebibliography}{10}

\bibitem{bgg}
Richard Beals, Bernard Gaveau, and Peter~C. Greiner.
\newblock Hamilton-{J}acobi theory and the heat kernel on {H}eisenberg groups.
\newblock {\em J. Math. Pures Appl. (9)}, 79(7):633--689, 2000.

\bibitem{blu-book}
A.~Bonfiglioli, E.~Lanconelli, and F.~Uguzzoni.
\newblock {\em Stratified {L}ie groups and potential theory for their
  sub-{L}aplacians}.
\newblock Springer Monographs in Mathematics. Springer, Berlin, 2007.

\bibitem{calin-book}
Ovidiu Calin, Der-Chen Chang, and Peter Greiner.
\newblock {\em Geometric analysis on the {H}eisenberg group and its
  generalizations}, volume~40 of {\em AMS/IP Studies in Advanced Mathematics}.
\newblock American Mathematical Society, Providence, RI, 2007.

\bibitem{calin-H-type}
Ovidiu Calin, Der-Chen Chang, and Irina Markina.
\newblock Geometric analysis on {$H$}-type groups related to division algebras.
\newblock Preprint, available online
  {\mbox{http://math.cts.nthu.edu.tw/Mathematics/preprints/prep2007-1-004.pdf}%
}, 2007.

\bibitem{cygan}
Jacek Cygan.
\newblock Heat kernels for class {$2$} nilpotent groups.
\newblock {\em Studia Math.}, 64(3):227--238, 1979.

\bibitem{davies-pang}
E.~B. Davies and M.~M.~H. Pang.
\newblock Sharp heat kernel bounds for some {L}aplace operators.
\newblock {\em Quart. J. Math. Oxford Ser. (2)}, 40(159):281--290, 1989.

\bibitem{eckmann}
Beno Eckmann.
\newblock Gruppentheoretischer {B}eweis des {S}atzes von {H}urwitz-{R}adon
  \"uber die {K}omposition quadratischer {F}ormen.
\newblock {\em Comment. Math. Helv.}, 15:358--366, 1943.

\bibitem{garofalo-segala}
Nicola Garofalo and Fausto Seg{\`a}la.
\newblock Estimates of the fundamental solution and {W}iener's criterion for
  the heat equation on the {H}eisenberg group.
\newblock {\em Indiana Univ. Math. J.}, 39(4):1155--1196, 1990.

\bibitem{gaveau77}
Bernard Gaveau.
\newblock Principe de moindre action, propagation de la chaleur et estim\'ees
  sous elliptiques sur certains groupes nilpotents.
\newblock {\em Acta Math.}, 139(1-2):95--153, 1977.

\bibitem{hormander67}
Lars H{\"o}rmander.
\newblock Hypoelliptic second order differential equations.
\newblock {\em Acta Math.}, 119:147--171, 1967.

\bibitem{hueber-muller}
H.~Hueber and D.~M{\"u}ller.
\newblock Asymptotics for some {G}reen kernels on the {H}eisenberg group and
  the {M}artin boundary.
\newblock {\em Math. Ann.}, 283(1):97--119, 1989.

\bibitem{hulanicki}
A.~Hulanicki.
\newblock The distribution of energy in the {B}rownian motion in the {G}aussian
  field and analytic-hypoellipticity of certain subelliptic operators on the
  {H}eisenberg group.
\newblock {\em Studia Math.}, 56(2):165--173, 1976.

\bibitem{hunt}
G.~A. Hunt.
\newblock Semi-groups of measures on {L}ie groups.
\newblock {\em Transactions of the American Mathematical Society},
  81(2):264--293, March 1956.

\bibitem{jerison-sanchez}
David~S. Jerison and Antonio S{\'a}nchez-Calle.
\newblock Estimates for the heat kernel for a sum of squares of vector fields.
\newblock {\em Indiana Univ. Math. J.}, 35(4):835--854, 1986.

\bibitem{kaplan80}
Aroldo Kaplan.
\newblock Fundamental solutions for a class of hypoelliptic {PDE} generated by
  composition of quadratic forms.
\newblock {\em Trans. Amer. Math. Soc.}, 258(1):147--153, 1980.

\bibitem{klingler}
Andrew Klingler.
\newblock New derivation of the {H}eisenberg kernel.
\newblock {\em Comm. Partial Differential Equations}, 22(11-12):2051--2060,
  1997.

\bibitem{kusuoka-stroock-III}
S.~Kusuoka and D.~Stroock.
\newblock Applications of the {M}alliavin calculus. {III}.
\newblock {\em J. Fac. Sci. Univ. Tokyo Sect. IA Math.}, 34(2):391--442, 1987.

\bibitem{levy50}
Paul L{\'e}vy.
\newblock Wiener's random function, and other {L}aplacian random functions.
\newblock In {\em Proceedings of the Second Berkeley Symposium on Mathematical
  Statistics and Probability, 1950}, pages 171--187, Berkeley and Los Angeles,
  1951. University of California Press.

\bibitem{li-jfa}
Hong-Quan Li.
\newblock Estimation optimale du gradient du semi-groupe de la chaleur sur le
  groupe de {H}eisenberg.
\newblock {\em J. Funct. Anal.}, 236(2):369--394, 2006.

\bibitem{li-heatkernel}
Hong-Quan Li.
\newblock Estimations asymptotiques du noyau de la chaleur sur les groupes de
  {H}eisenberg.
\newblock {\em C. R. Math. Acad. Sci. Paris}, 344(8):497--502, 2007.

\bibitem{magnus}
Wilhelm Magnus, Fritz Oberhettinger, and Raj~Pal Soni.
\newblock {\em Formulas and theorems for the special functions of mathematical
  physics}.
\newblock Third enlarged edition. Die Grundlehren der mathematischen
  Wissenschaften, Band 52. Springer-Verlag New York, Inc., New York, 1966.

\bibitem{montgomery}
Richard Montgomery.
\newblock {\em A tour of subriemannian geometries, their geodesics and
  applications}, volume~91 of {\em Mathematical Surveys and Monographs}.
\newblock American Mathematical Society, Providence, RI, 2002.

\bibitem{randall}
Jennifer Randall.
\newblock The heat kernel for generalized {H}eisenberg groups.
\newblock {\em J. Geom. Anal.}, 6(2):287--316, 1996.

\bibitem{rothschild-stein}
Linda~Preiss Rothschild and E.~M. Stein.
\newblock Hypoelliptic differential operators and nilpotent groups.
\newblock {\em Acta Math.}, 137(3-4):247--320, 1976.

\bibitem{taylor}
Thomas Taylor.
\newblock A parametrix for step-two hypoelliptic diffusion equations.
\newblock {\em Trans. Amer. Math. Soc.}, 296(1):191--215, 1986.

\bibitem{varopoulos-II}
N.~Th. Varopoulos.
\newblock Small time {G}aussian estimates of heat diffusion kernels. {II}.
  {T}he theory of large deviations.
\newblock {\em J. Funct. Anal.}, 93(1):1--33, 1990.

\bibitem{purplebook}
N.~Th. Varopoulos, L.~Saloff-Coste, and T.~Coulhon.
\newblock {\em Analysis and geometry on groups}, volume 100 of {\em Cambridge
  Tracts in Mathematics}.
\newblock Cambridge University Press, Cambridge, 1992.

\end{thebibliography}

\end{document}